\documentclass{mathSci}
\usepackage{graphicx}
\RequirePackage{amssymb}
\usepackage{amscd,latexsym,amsthm,amsfonts,amssymb,amsmath,amsxtra}
\usepackage[colorlinks=true,urlcolor=blue,citecolor=blue,linkcolor=black]{hyperref}
\usepackage{color}
\usepackage[all]{xy}
\usepackage[T1]{fontenc}
\usepackage{bm}
\usepackage{mathtools}
\usepackage{mathrsfs}
\usepackage{xcolor}
\usepackage{marginnote}
\usepackage{enumitem}

\let\Re\undefined

\DeclareMathOperator{\Re}{Re}

\DeclareMathOperator{\Tr}{Tr}

\DeclareMathOperator{\supp}{supp}

\newcommand{\R}{\mathbb{R}}

\newcommand{\Q}{\mathbb{Q}}

\newcommand{\A}{\mathbb{A}}

	\newcommand{\Res}{\operatorname{Res}}
	
	\newcommand{\Nr}{\operatorname{Nr}_{F/\Q}}

	\newcommand{\K}{\operatorname{K}}

	\newcommand{\Reg}{\operatorname{Reg}}

	\newcommand{\fin}{\operatorname{fin}}
	\newcommand{\diag}{\operatorname{diag}}

	\newcommand{\Vol}{\operatorname{Vol}}

	\newcommand{\sm}{\operatorname{Small}}

	\newcommand{\eps}{\varepsilon}

	\newcommand{\RNum}[1]{\uppercase\expandafter{\romannumeral #1\relax}}

\theoremstyle{plain}
	\newtheorem{thm}{Theorem}[section]
	
	\newtheorem{cor}[thm]{Corollary}
	\newtheorem{thmy}{Theorem}
	
	\newenvironment{thmx}{\stepcounter{thm}\begin{thmy}}{\end{thmy}}
	\newtheorem{cory}{Corollary}

	\theoremstyle{mylemmastyle}
	\newtheorem{lemma}[thm]{Lemma}
	\theoremstyle{plain}
	\newtheorem{prop}[thm]{Proposition}

	\theoremstyle{definition}

	\theoremstyle{remstyle}
	\newtheorem{remark}[thm]{Remark}

	\numberwithin{equation}{section}

\begin{document}
	
\title{Relative Trace Formula and Simultaneous Nonvanishing for $\mathrm{GL}_3\times \mathrm{GL}_1$ and $\mathrm{GL}_3\times \mathrm{GL}_2$ $L$-functions}%

\author{Philippe Michel$^{1,*}$ \and  Dinakar Ramakrishnan$^2$ \and Liyang Yang$^3$}

\affilOne{$^1$EPFL-SB-MATH-TAN Station 8,	1015 Lausanne, Switzerland\\
$^*$Corresponding Author. E-mail: philippe.michel@epfl.ch\\
}

\affilTwo{$^2$253-37 Caltech, Pasadena
	CA 91125, USA\\}
	
\affilThree{$^3$ Department of Mathematics, Princeton University,
	Princeton, NJ 08544, USA\\}

\maketitle
\begin{flushright}
  \textit{Dedicated to Harish Chandra, with admiration, on the centenary of his birth}
\end{flushright}

\vskip 0.5cm

\begin{abstract}
Fix a Dirichlet character $\chi$ and a cuspidal GL$(2)$ eigenform $\phi$ with relatively prime conductors. Then we show that there are infinitely many cusp forms $\pi$ on GL$(3)/\Q$ such that $L(1/2, \pi \times \chi)$ and $L(1/2, \pi \times \phi)$ are simultaneously non-zero. We achieve this by use of Jacquet's Relative Trace Formula. We derive an expression of the average over the GL$(3)$ cuspidal spectrum as a sum of a non-zero main term and two subsidiary terms which are forced to be zero for large enough level by use of a suitable test function. 
\end{abstract}

\subjClass{Primary: 11F67, 11F72; Secondary: 11F55, 11F66, 11F70, 22E55}

\tableofcontents

\setcounter{secnumdepth}{3}

\section{Introduction}

Harish Chandra was a mathematical giant who revolutionized the theory of representations of semisimple Lie groups $G$, which many of us utilize in several different facets of modern mathematics, from number theory at one end to geometry at the other. The basic objects of the theory are the discrete series whose criterion for existence (that $G$ possesses a compact Cartan subgroup) and whose mirific properties were established by him. They lead directly to tempered representations via parabolic induction of discrete series, and then on to the general Langlands classification and admissible $({\rm Lie}(G), K)$-modules. Of particular interest to us three (writing this article) is the unitary representation $L^2(\Gamma\backslash G)$ for a discrete subgroup $\Gamma$ of congruence type, where $G$ acts by right translation $\rho$. The uninitiated could first look at the case when $\Gamma$ is cocompact, in which case $L^2(\Gamma\backslash G)$ decomposes discretely into a Hilbert direct sum of irreducible unitary representations of $G$, each occurring with finite multiplicity. The simplest case is when $G$ is SL$(2,\R)$ and $\Gamma$ the group of integral units in an indefinite quaternion division algebra $B$ over $\Q$.

We are especially interested in the case when $G$ is $\mathrm{PGL}_3(\R)$, which already carries some of the difficulties encountered beyond $\mathrm{SL}_2(\R)$. We are in particular concerned with the periods over cycles defined by (discrete quotients of) suitable Lie subgroups $H_1$ and $H_2$, which are intimately related to the central values of Rankin-Selberg $L$-functions of GL$(3) \times {\rm GL}(1)$ and GL$(3) \times {\rm GL}(2)$; this relationship follows from the integral representations of Jacquet, Piatetski-Shapiro and Shalika in these two (very different) settings \cite{JPSS1,JPSS2}. We utilize for their study the relative trace formula of $G$ (acting on $L^2(\Gamma\backslash G)$) with respect to $(H_1, H_2)$. The spectral side furnishes the periods, and the essential difficulty of the paper is in analyzing the geometric side, isolating the main term and showing the remaining terms to be negligible.

We will be interested in congruence subgroups $\Gamma$, which are not cocompact. But we are able to move to the adelic picture, and work with automorphic representations $\pi$ of GL$_3(\A)$, with $\A$ being the adele ring of $\Q$, and by a suitable choice of a test function $f= \prod_v f_v$, restrict to those $\pi$ which have a supercuspidal component at a fixed prime $v_0$, which simplifies the trace formula and allows one to focus just on the discrete part, which is a big help. We will actually consider the periods weighted by a fixed character $\chi$ of $H_1$ and a fixed cusp form $\sigma$ on $H_2 = {\rm PGL}(2)$. The final objective is to show that for infinitely many cusp forms $\pi$ of GL$(3)$, $L(1/2, \pi \otimes \chi)$ and $L(1/2, \pi \times \sigma)$ are simultaneously non-zero.

The second author of this article (D.R.) gave a lecture at the Harish Chandra Centennial conference at HRI in Allahabad in October, 2023, when he focused on a similar, but much more difficult case of $U(2,1) \times U(1,1)$ (see \cite{MRY23}). The case considered here is much shorter and is more amenable to an easier exposition. We hope it will be of some use.

Let us begin with a quick review of Jacquet's relative trace formula (or RTF, for short), which the experts can skip. For simplicity let us pretend that $\Gamma\backslash G$ is compact. Every smooth compactly supported test function $f$ on $G$ acts as a nice compact operator on $L^2(\Gamma\backslash G)$ by $\rho(f)\varphi: = \varphi \mapsto \int_{\Gamma\backslash G} f(g)\rho(g)\varphi dg$, where $dg$ is the quotient measure induced by the Haar measure on $G$. Since $\varphi$ is left invariant under $\Gamma$, $\rho(f)\varphi$ is the pairing $\langle \varphi, {K_f}_{|_{\Delta}}\rangle_{\Gamma\backslash G}$, where $K_f$ is the kernel function on $\Gamma\backslash G \times \Gamma\backslash G$ given by $(x,y) \mapsto \sum_{\gamma \in \Gamma} \, f(x^{-1}\gamma y)$, and $\Delta$ denotes the diagonally embedded $\Gamma\backslash G$. This leads to the expression of the trace of $f$ as the integral of $K_f$ over $\Delta$. On the one hand, a {\it spectral analysis} of $L^2(\Gamma\backslash G)$ as $\sum_{\pi \in \hat G} \, m_\pi V_\pi$ into a sum of irreducible unitary representations $(\pi, V_\pi)$ with multiplicity $m_\pi$. On the other hand, the geometric side expands $\int_\Delta K_f$ as $\sum_{\{\gamma\}} {\rm vol}(\Gamma_\gamma\backslash G_\gamma) O_\gamma(f)$, where $\{\gamma\}$ is the set of conjugacy classes of $\gamma$ (with representative $\gamma$), $G_\gamma$ (resp. $\Gamma_\gamma$) is the centralizer of $\gamma$ in $G$ (resp. $\Gamma$), ${\rm vol}(\Gamma_\gamma\backslash G_\gamma)$ is the volume of the quotient space of $G_\gamma$ under the left action of $\Gamma_\gamma$, and $O_\gamma(f)$ denotes the {\it orbital integral} $\int_{G_\gamma\backslash G} f(x^{-1}\gamma y) dx$. Comparing the two sides one often gets many beautiful consequences. When $f$ is a Hecke operator, the geometric picture to keep in mind is the intersection of the graph of the Hecke correspondence with the diagonal. Jacquet's ingenious idea was to replace the diagonal by a product $Z_1 \times Z_2$ of suitable codimension  in $(\Gamma\backslash G)^2$ (so it is not too big or too small), and consider the integral - ``{\it period}'' - of the kernel $K_f$ over $Z_1 \times Z_2$. To be able to apply representation theory, one is led to restrict the ``{\it cycles}'' $Z_j, j=1, 2$ to be arithmetic quotients $\Gamma'_j\backslash H_j$, with $H_j$ Lie subgroups of $G$ and $\Gamma'_j = \Gamma \cap H_j$; note that the $H_j$ are not necessarily semisimple. A simple example is $G=\mathrm{PGL}(2,\R)$, $\Gamma$ cocompact defined by an indefinite quaternion division algebra $B$ over $\Q$, $\Gamma\backslash G$ a circle bundle over the Riemann surface $\Gamma\backslash \mathcal{H}$ (where $\mathcal{H}$ is the upper half plane) and $H_j$ is (for $j=1,2$) the torus defined by the algebraic integers of norm one in a real quadratic field $F_j$ where $B$ splits; here the image of $Z_1\times Z_2$ in $(\Gamma\backslash \mathcal{H})^2$ is a product of two circles (and has the same dimension as the diagonal Riemann surface).

Often one also integrates against a weight function $(\xi_1, \xi_2)$ on $Z_1 \times Z_2$. The kernel function $K_f$ can be expanded spectrally as $\sum_r \frac{\rho(f)\varphi_r \otimes \overline \varphi_r}{\langle \varphi_r, \varphi_r\rangle}$, where the sum is over an orthogonal basis $\{\varphi_r | r \in {\mathbb N}\}$ of the automorphic spectrum. This leads to the spectral side of the relative trace formula for $\int\int_{Z_1 \times Z_2} \, K_f(x,y) \xi_1(x)\xi_2(y)dxdy$ to be $\sum_r {\mathcal P_1(\rho(f)\varphi_r, \xi_1)} {\mathcal P_2(\overline \varphi_r, \xi_2)}$, where $\mathcal P_1(\rho(f)\varphi_r, \xi_1)$, resp. $\mathcal P_2(\overline\varphi_r, \xi_2)$, equals the {\it weighted period} $\int_{Z_1} \rho(f)(\varphi_r)(x) \xi_1(x) dx$, resp. $\int_{Z_2} {\overline\varphi_r}(y) \xi_2(y) dy$. In some cases, $\rho(f)$ acts as a scalar $\lambda_r$ on $\varphi_r$, simplifying the first period. A miracle is that in certain cases, the period integrals are related to $L$-values, as it is in the case we analyze in this article. (For a discussion of the simpler case of PGL$(2)/\Q$ with $H_1=H_2=T$, the diagonal torus, see \cite{RR05}.) Writing $H=H_1 \times H_2$ and $\Gamma'= \Gamma'_1 \times \Gamma'_2$, we may expand the geometric side as $\sum_{\gamma \in \Gamma /\sim} {\rm vol}(\Gamma'_\gamma\backslash H_\gamma) O_\gamma(f, \xi)$, where $\Gamma/\sim$ is the set of equivalence classes of $\gamma \in \Gamma$ with the equivalence $\gamma \equiv \delta$ (in $\Gamma$) iff $\delta = h_1\gamma h_2$ for some $h_1, h_2$ , $H_\gamma$ (resp. $\Gamma'_\gamma$) is the stabilizer of $\gamma$ in $H$ (resp. $\Gamma'$), and $O_\gamma(f, \xi)$ denotes the integral $\int_{H_\gamma\backslash H} f(x^{-1}\gamma y)\xi_1(x)\overline\xi_2(y)dx dy$.

In our case of interest, where $G=$PGL$(3)$, we take $H_2$ to be the image of GL$(2)$ in  $G$ via the embedding $g\mapsto \diag(g,1)$, and the period is given (thanks to \cite{JPSS1}) essentially by $L(1/2, \varphi_r \times \sigma)$ if the weight function on $H_2$ is given by a cusp form in the space of $\sigma$. Even more interesting is the fact that if we take $H_1$ to be $\big\{\begin{pmatrix}
a& 0 & b\\
& 1& c\\
&& 1
\end{pmatrix} \big\}$, whose abelianization is $\mathrm{GL}(1)$, then the corresponding period is (essentially) $L(1/2, \varphi_r \times \chi)$, for $\xi_1$ a character $\chi$ (\cite{JPSS2}). And our object is to prove that for many $\varphi_r$ belonging to different cuspidal automorphic forms on $G$, these two periods are simultaneously non-zero. We can achieve this because there are, for our choice of a test function $f$, three terms on the geometric side, and two of those happen to be zero (after some work) by suitably shrinking the support of $f$. The remaining (main) term is shown to be non-zero, as the (prime) level of the GL$(3)$ forms becomes large.
We have not tried here to quantify here the amount of non-vanishing, as have done in \cite{MRY23}.

\medskip

We conclude the Introduction with a few comments. In our main result below, it should be possible with some work to deal with the more general case where we do {\it not} require (by our choice of $f$) any supercuspidal component of $\pi$, but it will  be more delicate. Much harder will be the case of $\pi'$ (on GL$(2)$) {\it not} cuspidal, which will require a subtle regularization and consequently some additional terms, and it will be impressive to solve this case. In a different direction, one could try to get the second moment of GL$(3) \times {\rm GL}(1)$ $L$-functions, where we will take $H_2$ to be the same as the $H_1$ used here; the geometric side is quite difficult in this case. By contrast, the case of the second moment of the GL$(3) \times {\rm GL}(2)$ $L$-functions works out a bit better, though still difficult; in fact the third author of this paper has achieved that for GL$(n+1) \times {\rm GL}(n)$ $L$-functions (\cite{Y23}) for any $n \geq 2$.

\medskip

\section{The Main Result}

\begin{thmx}\label{M}
Let $\chi$ be a unitary Hecke character of $F^{\times}\backslash\mathbb{A}_F^{\times}$. Let $\sigma$ be a unitary cuspidal automorphic representation of $\mathrm{GL}_2(\A_F)$. Then there are infinitely many unitary cuspidal automorphic representations $\pi$ of $\mathrm{PGL}_3(\A_F)$ such that
\begin{equation}\label{non-vanish}
L(1/2,\pi\times\chi)L(1/2,\pi\times\sigma)\neq 0.
\end{equation}

\end{thmx}

\begin{remark}
	More precisely, we will prove that \eqref{non-vanish} holds for all $\pi$ such that
	\begin{itemize}
		\item $\pi_\infty$ is contained in a compact subset of the unitary dual of $\mathrm{GL}_3(F_\infty)$,
		\item there is a fixed finite split  place $v_0$, at which $\pi_{v_0}$ is isomorphic to a (fixed) supercuspidal representation of $\mathrm{PGL}_3(F_{v_0})$,
		\item $\pi$ has an Iwahori invariant vector at some prime ideal $\mathfrak{N}$ with $\Nr(\mathfrak{N})$ sufficiently large (in particular distinct from $v_0$),
		\item $\pi$ is unramified at every other finite place.
	\end{itemize}
	For $\Nr(\mathfrak{N})$ large enough the number of such $\pi$ is $\asymp \Nr(\mathfrak{N})^3$. One may then ask how many such $\pi$ satisfy \eqref{non-vanish}. In \cite{MRY23} we could establish that this was the case for $\gg \Nr(\mathfrak{N})^\delta$ for some fixed $\delta>0$. The proof used the {\em amplification method} to bound individually  the relevant  period integrals; by design these were {\em non-negative}. In the present case, we lack this positivity property. Instead, using the {\em mollification method} one could show that \eqref{non-vanish} holds for $\gg\log(\Nr(\mathfrak{N}))^\delta$ for some fixed $\delta>0$. We leave this to the interested reader.
	  \end{remark}

\medskip

\subsection{Notation Guide}
\subsubsection{Number Fields and Measures}\label{1.1.1}
Let $F$ be a number field with ring of integers $\mathcal{O}_F.$ Let $N_F$ be the absolute norm. Let $\mathfrak{O}_F$ be the different of $F.$ Let $\mathbb{A}_F$ be the adele group of $F.$ Let $\Sigma_F$ be the set of places of $F.$ Denote by $\Sigma_{F,\fin}$ (resp. $\Sigma_{F,\infty}$) the set of non-Archimedean (resp. Archimedean) places. For $v\in \Sigma_F,$ we denote by $F_v$ the corresponding local field. For a non-Archimedean place $v,$ let $\mathcal{O}_v$ be the ring of integers of $F_v$, and $\mathfrak{p}_v$ be the maximal prime ideal in $\mathcal{O}_v$. Given an integral ideal $\mathcal{I},$ we say $v\mid \mathcal{I}$ if $\mathcal{I}\subseteq \mathfrak{p}_v.$ Fix a uniformizer $\varpi_{v}\in\mathfrak{p}_v.$ Denote by $e_v(\cdot)$ the evaluation relative to $\varpi_v$ normalized as $e_v(\varpi_v)=1.$ Let $q_v$ be the cardinality of $\mathbb{F}_v:=\mathcal{O}_v/\mathfrak{p}_v.$ We use $v\mid\infty$ to indicate an Archimedean place $v$ and write $v<\infty$ if $v$ is non-Archimedean. Let $|\cdot|_v$ be the norm in $F_v.$ Put $|\cdot|_{\infty}=\prod_{v\mid\infty}|\cdot|_v$ and $|\cdot|_{\fin}=\prod_{v<\infty}|\cdot|_v.$ Let $|\cdot|_{\mathbb{A}_F}=|\cdot|_{\infty}\otimes|\cdot|_{\fin}$. We will simply write $|\cdot|$ for $|\cdot|_{\mathbb{A}_F}$ in calculation over $\mathbb{A}_F^{\times}$ or its quotient by $F^{\times}$.

Let $\psi_{\mathbb{Q}}$ be the additive character on $\mathbb{Q}\backslash \mathbb{A}_{\mathbb{Q}}$ such that $\psi_{\mathbb{Q}}(t_{\infty})=\exp(2\pi it_{\infty}),$ for $t_{\infty}\in \mathbb{R}\hookrightarrow\mathbb{A}_{\mathbb{Q}}.$ Let $\psi_F=\psi_{\mathbb{Q}}\circ \Tr_F,$ where $\Tr_F$ is the trace map. Then $\psi_F(t)=\prod_{v\in\Sigma_F}\psi_v(t_v)$ for $t=(t_v)_v\in\mathbb{A}_F.$ For $v\in \Sigma_F,$ let $dt_v$ be the additive Haar measure on $F_v,$ self-dual relative to $\psi_v.$ Then $dt=\prod_{v\in\Sigma_F}dt_v$ is the standard Tamagawa measure on $\mathbb{A}_F$. Let $d^{\times}t_v=\zeta_{F_v}(1)dt_v/|t_v|_v,$ where $\zeta_{F_v}(\cdot)$ is the local Dedekind zeta factor. In particular, $\Vol(\mathcal{O}_v^{\times},d^{\times}t_v)=\Vol(\mathcal{O}_v,dt_v)=N_{F_v}(\mathfrak{D}_{F})^{-1/2}$ for all finite place $v.$ Moreover, $\Vol(F\backslash\mathbb{A}_F; dt_v)=1$ and $\Vol(F\backslash\mathbb{A}_F^{(1)},d^{\times}t)=\underset{s=1}{\Res}\ \zeta_F(s),$ where $\mathbb{A}_F^{(1)}$ is the subgroup of ideles $\mathbb{A}_F^{\times}$ with norm $1,$ and $\zeta_F(s)=\prod_{v<\infty}\zeta_{F_v}(s)$ is the finite Dedekind zeta function. Denote by $\widehat{F^{\times}\backslash\mathbb{A}_F^{(1)}}$  the Pontryagin dual of $F^{\times}\backslash\mathbb{A}_F^{(1)}.$
	
\subsubsection{Reductive Groups}
For an algebraic group $H$ over $F$, we will denote by $[H]:=H(F)\backslash H(\mathbb{A}_F).$ We equip measures on $H(\mathbb{A}_F)$ as follows: for each unipotent group $U$ of $H,$ we equip $U(\mathbb{A}_F)$ with the Haar measure such that, $U(F)$ being equipped with the counting measure. The measure of $[U]$ is $1.$ We equip the maximal compact subgroup $K$ of $H(\mathbb{A}_F)$ with the Haar measure such that $K$ has total mass $1.$ When $H$ is split, we also equip the maximal split torus of $H$ with Tamagawa measure induced from that of $\mathbb{A}_F^{\times}.$

In this paper we set $G=\mathrm{GL}(3)$ and $G'=\mathrm{GL}_2$. Let $B$ be the group of upper triangular matrices in $G$.
Let $\overline{G}=Z\backslash G$ and $B_0=Z\backslash B,$ where $Z$ is the center of $G.$ Let $N$ be the unipotent radical of $B$. Let $\theta$ be the generic character induced by $\psi_F$:
\begin{equation}\label{eq1.12}
\theta\left(u\right)=\psi_F(u_{12}+u_{23}),\  \ u=(u_{ij})\in N(\mathbb{A}_F).
\end{equation}
Then $\theta=\otimes_v\theta_v$, where $\theta_v\left(u_v\right):=\psi_v(u_{12,v}+u_{23,v})$, $v\in \Sigma_F$.

Let $K'=\otimes_vK_v'$ be a maximal compact subgroup of $G'(\mathbb{A}_F),$ where $K_v'=\mathrm{U}_2(\mathbb{C})$ if $v$ is complex, $K_v'=\mathrm{O}_2(\mathbb{R})$ if $v$ is real, and $K_v'=G'(\mathcal{O}_v)$ if $v<\infty.$ For $v\in \Sigma_{F,\fin},$ $m\in\mathbb{Z}_{\geq 0},$ define
\begin{equation}\label{2.1}
K_v[m]:=\Big\{(k_{ij})_{1\leq i,j\leq 3}\in G(\mathcal{O}_v):\ k_{31}, k_{32}\in \mathfrak{p}_v^{m}\Big\}.
\end{equation}

Define the following matrices:
\begin{align*}
w_1 = \begin{pmatrix}
& 1 \\
1 & \\
&& 1
\end{pmatrix}, \quad
w_2 = \begin{pmatrix}
1 & \\
&& 1 \\
& 1 &
\end{pmatrix}.
\end{align*}

\subsubsection{Other Conventions}\label{sec1.5.4}
Denote by $\alpha\asymp \beta$ for $\alpha, \beta \in\mathbb{R}$ if there are absolute constants $c$ and $C$ such that $c\beta\leq \alpha\leq C\beta.$

Throughout the paper, we fix a number field $F$, functions  $h_v\in C_c^{\infty}(\mathrm{GL}_3(F_v))$ at $v\mid\infty$ supported in a small neighborhood of the identity $I_3$ with $h_v(I_3)=1$, a supercuspidal representation $\pi_{v_0}$, a unitary Hecke character $\chi$ of $F$, and a unitary cuspidal automorphic representation $\sigma$ of $\mathrm{GL}_2/F$ with a cusp form $\phi\in \sigma$. All implied constants in $\ll$ or $O(\cdot)$ will depend at most on $F$, $h_v$'s, $\pi_{v_0}$, $\chi$ and $\phi$.


\section{Test Functions and the Two Sides of RTF}
\subsection{Intrinsic Data}
\subsubsection{Automorphic Representations}
Let $\chi=\otimes_v\chi_v$ be a unitary Hecke character of $F^{\times}\backslash\mathbb{A}_F^{\times}$. Let $\sigma=\otimes_v\sigma_v$ be a unitary automorphic representation of $\mathrm{GL}_2/F$.

For $v<\infty$, we denote by  $r_{\chi_v}$ and $r_{\sigma_v}$ the exponent of the arithmetic conductor of $\chi_v$ and $\sigma_v$, respectively. Let
\begin{itemize}
\item $\Sigma_{\chi_{\fin}}:=\big\{v<\infty:\ r_{\chi_v}>0\big\}$, the set of ramified places of $\chi_{\fin}=\otimes_{v<\infty}\chi_v$.
\item $\Sigma_{\sigma_{\fin}}:=\big\{v<\infty:\ r_{\sigma_v}>0\big\}$, the set of ramified places of $\sigma_{\fin}=\otimes_{v<\infty}\sigma_v$.
\end{itemize}
Then $\Sigma_{\chi_{\fin}}$ and $\Sigma_{\sigma_{\fin}}$ are finite sets. 

\subsubsection{Cusp Forms in $\sigma$}\label{sec2.1.2}
For $v\mid\infty$, let $\alpha_v\in C_c^{\infty}(F_v^{\times})$ be such that $\alpha_v(1)=1$ and $\alpha_v(t_v)=0$ if $|t_v-1|> \varepsilon$. Let $W_v$ be the function in the Whittaker model satisfying
\begin{align*}
W_v\begin{pmatrix}
a_v\\
& 1
\end{pmatrix}=\alpha_v(a_v),\ \ a_v\in F_v^{\times}.
\end{align*}

For $v\leq \infty$, we let $W_v$ be the local new vector in the Whittaker model of $\sigma_v$, normalized as $W_v(I_2)=1$.

Let $\phi$ be the cusp form in $\sigma$ corresponding to the Whittaker vector $\otimes_{v\leq \infty}W_v$.

\subsubsection{Ramifications}\label{sec2.1.3}
Let $\mathfrak{N}\subsetneq \mathcal{O}_F$ be a prime ideal. Let
\begin{equation}\label{equ2.1}
I(\mathfrak{N}):=I_{\mathfrak{N}}[1]\prod_{v<\infty,\ v\nmid\mathfrak{N}}K_v,
\end{equation}
where $I_{\mathfrak{N}}[1]$ is the Iwahori subgroup of level $1$ at $v=\mathfrak{N}$.
Suppose the absolute norm $N_F(\mathfrak{N})$ is sufficiently large so that $\chi_v$ and $\pi_v$ are unramified at $v=\mathfrak{N}$.

\subsubsection{Supercuspidal Components}
Let $v_0$ be a fixed finite split place such that $r_{\chi_{v_0}}=r_{\sigma_{v_0}}=0$ and $v_0\nmid\mathfrak{N}$. Let $\pi_{v_0}$ be a unitary supercuspidal representation of $\mathrm{PGL}_3/F_{v_0}$. Let $W_{v_0}$ be the vector in the Whittaker model of $\pi_{v_0}$ (relative to $\overline{\psi}_{v_0}$) such that
\begin{equation}\label{mc}
W_{v_0}\begin{pmatrix}
y_{v_0}\\
& 1
\end{pmatrix}=\mathbf{1}_{K_{v_0}'^{\circ}[1]}(y_{v_0}),\ \ y_{v_0}\in \mathrm{GL}_2(F_{v_0}).
\end{equation}
Here $K_{v_0}'^{\circ}[1]$ is the congruence subgroup defined by
\begin{align*}
K_{v_0}'^{\circ}[1]:=\bigg\{\begin{pmatrix}
k_{11}& k_{12}\\
k_{21} & k_{22}
\end{pmatrix}:\ \ k_{21}\in \mathfrak{p}_{v_0},\ k_{22}-1\in \mathfrak{p}_{v_0}\bigg\}.
\end{align*}

By Kirillov model theory, such a Whittaker vector $W_{v_0}$ is uniquely determined by \eqref{mc}. Moreover, we have the following property of the matrix coefficient.
\begin{lemma}\label{lem2.1}
Let notation be as before. Let $y_{v_0}:=\begin{pmatrix}
1& & b_{v_0}\\
& 1& c_{v_0}\\
&& 1
\end{pmatrix}\begin{pmatrix}
x_{v_0}\\
& 1
\end{pmatrix}$, where $b_{v_0}, c_{v_0}\in F_{v_0}$, and $x_{v_0}\in \mathrm{GL}_2(F_{v_0})$. Then
\begin{equation}\label{eq2.2}
\langle \pi_{v_0}(y_{v_0})W_{v_0},W_{v_0}\rangle=\overline{\psi}_{v_0}(c_{v_0})\Vol(K_{v_0}'^{\circ}[1])\mathbf{1}_{b_{v_0}, c_{v_0}\in \mathfrak{p}_{v_0}^{-1}}\mathbf{1}_{K_{v_0}'^{\circ}[1]}(x_{v_0}).	
\end{equation}
\end{lemma}
\begin{proof}
By definition, the matrix coefficient is
\begin{align*}
\langle \pi_{v_0}(y_{v_0})W_{v_0},W_{v_0}\rangle=\int_{N(F_{v_0})\backslash\mathrm{GL}_v(F_{v_0})}W_{v_0}\left(\begin{pmatrix}
g\\
& 1
\end{pmatrix}y_{v_0}\right)\overline{W_{v_0}\begin{pmatrix}
g\\
& 1
\end{pmatrix}}|\det g|_vdg.
\end{align*}

Substituting \eqref{mc} into the above integral yields
\begin{equation}\label{eq2.3}
\langle \pi_{v_0}(y_{v_0})W_{v_0},W_{v_0}\rangle=\int_{K_{v_0}'^{\circ}[1]}W_{v_0}\left(\begin{pmatrix}
k\\
& 1
\end{pmatrix}\begin{pmatrix}
1& & b_{v_0}\\
& 1& c_{v_0}\\
&& 1
\end{pmatrix}\begin{pmatrix}
x_{v_0}\\
& 1
\end{pmatrix}\right)dk.
\end{equation}

Write $k=\begin{pmatrix}
k_{11} & k_{12}\\
k_{21} & k_{22}
\end{pmatrix}\in K_{v_0}'^{\circ}[1]$. Then it follows from \eqref{eq2.3} that
\begin{equation}\label{eq2.4}
\langle \pi_{v_0}(y_{v_0})W_{v_0},W_{v_0}\rangle=\int_{K_{v_0}'^{\circ}[1]}\overline{\psi}_{v_0}(k_{21}b_{v_0}+k_{22}c_{v_0})W_{v_0}\begin{pmatrix}
kx_{v_0}\\
& 1
\end{pmatrix}dk.
\end{equation}

Notice that $W_{v_0}\begin{pmatrix}
kx_{v_0}\\
& 1
\end{pmatrix}=\mathbf{1}_{K_{v_0}'^{\circ}[1]}(x_{v_0})$. We thus derive \eqref{eq2.2} from \eqref{eq2.4}, along with orthogonality of the additive character $\overline{\psi}_{v_0}$.
\end{proof}

\subsection{Construction of the Test Function}\label{sec2.2}
\begin{itemize}
\item at $v=v_0$, we take
\begin{align*}
f_{v_0}(g_{v_0}):=\Vol(K_{v_0}'^{\circ}[1])^{-2}\langle \pi_{v_0}(g_{v_0})W_{v_0},W_{v_0}\rangle,\ \ g_{v_0}\in \mathrm{GL}_3(F_{v_0}).
\end{align*}
\item at $v\mid\mathfrak{N}$, we take
\begin{align*}
f_v(g_v):=\Vol(I_v[1])^{-1}\int_{F_v^{\times}}\mathbf{1}_{I_v[1]}(z_vg_v)d^{\times}z_v,\ \ g_v\in \mathrm{GL}_3(F_v),
\end{align*}
where $I_v[1]$ is the Iwahori subgroup of level $1$.

\item at $v\in \Sigma_{\chi_{\fin}}\bigcup \Sigma_{\sigma_{\fin}}$, we take
\begin{align*}
f_v(g_v):=&\frac{1}{\tau(\chi_v)}\sum_{\alpha_v\in (\mathcal{O}_v/\mathfrak{p}_v^{r_{\chi_v}})^{\times}}\chi_v(\alpha_v)\sum_{\beta_v\in \mathcal{O}_v/\mathfrak{p}_v^{r_{\sigma_v}}}\int_{F_v^{\times}}\\
&\qquad \qquad (\mathbf{1}_{U_1(\alpha_v\varpi_v^{-r_{\chi_v}})K_v}*\mathbf{1}_{U_2(\beta_v\varpi_v^{-r_{\sigma_v}})K_v})(z_vg_v)d^{\times}z_v,
\end{align*}
where $\tau(\chi_v)$ is the Gauss sum relative to $\chi_v$ and $\psi_v$, and  for $\gamma_v\in F_v$,
\begin{equation}\label{eq2.5}
U_1(\gamma_v):=\begin{pmatrix}
1&\gamma_v & \\
&1 \\
&& 1
\end{pmatrix},\ \ \ U_2(\gamma_v):=\begin{pmatrix}
1& & \gamma_v\\
&1 \\
&& 1
\end{pmatrix}.
\end{equation}

\item at the remaining finite places $v$, we take
\begin{align*}
f_v(g_v):=\int_{F_v^{\times}}\mathbf{1}_{\mathrm{GL}_3(\mathcal{O}_v)}(z_vg_v)d^{\times}z_v,\ \ g_v\in \mathrm{GL}_3(F_v).
\end{align*}

\item at Archimedean places $v\mid\infty$ we take
\begin{equation}\label{eq2.6}
f_v(g_v):=\int_{F_v^{\times}}h_v(z_vg_v)d^{\times}z_v,\ \ g_v\in \mathrm{GL}_3(F_v),
\end{equation}
where $h_v\in C_c^{\infty}(\mathrm{GL}_3(F_v))$ is a fixed smooth function,  with $h_v(I_3)=1$, and supported in an $\eps$-neighborhood of the identity $I_3$ (ie. defined by $\|g_v\|_v\leq \eps$ for some fixed norm on $\mathrm{M}_3(F_v)$) with $\eps>0$ fixed sufficiently small.
\end{itemize}

Let $f=\otimes_{v\leq \infty}f_v$ be the test function constructed as above. Henceforth we will work below with this test function in the relative trace formula, or RTF for short.

\subsection{The Relative Trace Formula}\label{sec2.3}
Let $f$ be the test function defined in \textsection\ref{sec2.2}. Let
\begin{equation}\label{geom}
\K(g_1,g_2):=\sum_{\gamma\in \mathrm{PGL}_3(F)}f(g_1^{-1}\gamma g_2),\ \ g_1, g_2\in \mathrm{GL}_3(\mathbb{A}_F)
\end{equation}
be the automorphic kernel function associated with $f$. Since $f$ has compact support, $\K(g_1,g_2)$ converges absolutely.

Let $\mathbf{s}=(s_1,s_2)\in \mathbb{C}^2$ be such that $\Re(s_1)\gg 1$ and $\Re(s_2)\gg 1$. Let $\phi$ be an automorphic form in $\sigma$. Define
\begin{align*}
J(f,\phi,\boldsymbol{s}):=\int_{F^{\times}\backslash\mathbb{A}_F^{\times}} d^{\times}a \int_{(F\backslash\mathbb{A}_F)^2} db dc
\int_{[\mathrm{GL}_2]}&\K\left(\begin{pmatrix}
a& & b\\
& 1& c\\
&& 1
\end{pmatrix},\begin{pmatrix}
y&\\
& 1
\end{pmatrix}\right)\chi(a)\overline{\psi}(c)\\
& |a|^{s_1-\frac{1}{2}}\phi(y)|\det y|^{s_2}d y.
\end{align*}

\subsection{The Spectral Side}
Since $f_v$ is a matrix coefficient of a supercuspidal representation at $v=v_0$, only the cuspidal spectrum appears in the spectral decomposition of $\K(\cdot,\cdot)$. Hence,
\begin{equation}\label{f2.8}
\K(g_1,g_2)=\sum_{\pi}\sum_{\varphi\in \mathfrak{B}_{\pi}^{I(\mathfrak{N})}}\pi(f)\varphi(g_1)\overline{\varphi(g_2)},
\end{equation}
where $\mathfrak{B}_{\pi}$ is an orthonormal basis of $\pi$,  invariant under the action of $I(\mathfrak{N}):=\prod_{v<\infty}K_v[e_v(\mathfrak{N})]$. Substituting \eqref{f2.8} into the definition of $J(f,\phi,\boldsymbol{s})$ yields
\begin{equation}\label{f2.9}
J(f,\phi,\boldsymbol{s})=\sum_{\pi}\sum_{\varphi\in \mathfrak{B}_{\pi}^{I(\mathfrak{N})}}\mathcal{P}_1(1/2+s_1,\pi(f)\varphi,\chi)\mathcal{P}_2(1/2+s_2,\overline{\varphi},\phi),
\end{equation}
where $\mathcal{P}_1(1/2+s_1,\pi(f)\varphi,\chi)$ is defined by
\begin{align*}
\int_{F^{\times}\backslash\mathbb{A}_F^{\times}}\int_{(F\backslash\mathbb{A}_F)^2}\pi(f)\varphi\begin{pmatrix}
a& & b\\
& 1& c\\
&& 1
\end{pmatrix}\chi(a)\overline{\psi}(c)|a|^{s_1-\frac{1}{2}}dbdcd^{\times}a,
\end{align*}
and $\mathcal{P}_2(1/2+s_2,\overline{\varphi},\phi)$ is defined by
\begin{align*}
\int_{[\mathrm{GL}_2]}&\overline{\varphi\begin{pmatrix}
y&\\
& 1
\end{pmatrix}}\phi(y)|\det y|^{s_2}dy.
\end{align*}

Notice that $\mathcal{P}_1(1/2+s_1,\pi(f)\varphi,\chi)$ is the Rankin-Selberg integral representing the $L$-function $L(1/2+s_1,\pi\times\chi)$, and $\mathcal{P}_2(1/2+s_2,\overline{\varphi},\phi)$ represents $L(1/2+s_2,\widetilde{\pi}\times\sigma)$, where $\widetilde{\pi}$ is the contragredient of $\pi$.

Due to the rapid decay of cusp forms and Rankin-Selberg theory, the function $J(f,\phi,\boldsymbol{s})$ converges absolutely in $\Re(s_1)\gg 1$ and $\Re(s_2)\gg 1$, and admits a holomorphic continuation to $\mathbb{C}^2$.

\subsection{The Geometric Side}\label{sec2.5}
Let $P$ be the parabolic subgroup of $\mathrm{GL}_3$ of type $(2,1)$, and let $P_0$ be the mirabolic subgroup. Let $ N_{w_2} $ be the unipotent subgroup consisting of matrices of the form $\begin{pmatrix}
1 & & \\
& 1 & * \\
&& 1
\end{pmatrix}$, and let $N_{w_1w_2}$ be the unipotent subgroup consisting of matrices of the form $\begin{pmatrix}
1 & * & * \\
& 1 & \\
&& 1
\end{pmatrix}$.

Then, we have the Bruhat decomposition:
\begin{equation}\label{Bruhat}
\mathrm{GL}_3(F) = P(F) \sqcup N_{w_2}(F) w_2 P(F) \sqcup N_{w_1w_2}(F) w_1 w_2 P(F).
\end{equation}

Substituting \eqref{geom} and \eqref{Bruhat} into the definition of $J(f,\phi,\boldsymbol{s})$ yields
\begin{equation}\label{eq2.7}
J(f,\phi,\boldsymbol{s})=J_{\sm}(f,\phi,\boldsymbol{s})+J_{\Reg}^{\RNum{1}}(f,\phi,\boldsymbol{s})+J_{\Reg}^{\RNum{2}}(f,\phi,\boldsymbol{s}),
\end{equation}
where
\begin{align*}
&J_{\sm}(f,\phi,\boldsymbol{s}):=\int_{F^{\times}\backslash\mathbb{A}_F^{\times}}\int_{(F\backslash\mathbb{A}_F)^2}\int_{[\mathrm{GL}_2]}\sum_{\gamma\in P_0(F)}f\left(\begin{pmatrix}
a& & b\\
& 1& c\\
&& 1
\end{pmatrix}^{-1}\gamma\begin{pmatrix}
y&\\
& 1
\end{pmatrix}\right)\\
&\qquad \qquad \qquad \qquad \chi(a)\overline{\psi}(c)|a|^{s_1-\frac{1}{2}}\phi(y)|\det y|^{s_2}d y dbdcd^{\times}a,
\end{align*}
the function $J_{\Reg}^{\RNum{1}}(f,\phi,\boldsymbol{s})$ is defined by
\begin{align*}
&\int_{F^{\times}\backslash\mathbb{A}_F^{\times}}\int_{(F\backslash\mathbb{A}_F)^2}\int_{[\mathrm{GL}_2]}\sum_{\gamma\in N_{w_2}(F)w_2P_0(F)}f\left(\begin{pmatrix}
a& & b\\
& 1& c\\
&& 1
\end{pmatrix}^{-1}\gamma\begin{pmatrix}
y&\\
& 1
\end{pmatrix}\right)\\
&\qquad \qquad \qquad \qquad \chi(a)\overline{\psi}(c)|a|^{s_1-\frac{1}{2}}\phi(y)|\det y|^{s_2}d y dbdcd^{\times}a,
\end{align*}
and the function $J_{\Reg}^{\RNum{2}}(f,\phi,\boldsymbol{s})$ is defined by
\begin{align*}
&\int_{F^{\times}\backslash\mathbb{A}_F^{\times}}\int_{(F\backslash\mathbb{A}_F)^2}\int_{[\mathrm{GL}_2]}\sum_{\gamma\in N_{w_1w_2}(F)w_1w_2P_0(F)}f\left(\begin{pmatrix}
a& & b\\
& 1& c\\
&& 1
\end{pmatrix}^{-1}\gamma\begin{pmatrix}
y&\\
& 1
\end{pmatrix}\right)\\
&\qquad \qquad \qquad \qquad \chi(a)\overline{\psi}(c)|a|^{s_1-\frac{1}{2}}\phi(y)|\det y|^{s_2}d y dbdcd^{\times}a.
\end{align*}

We will refer to $J_{\sm}(f,\phi,\boldsymbol{s})$ as the small cell orbital integral, while $J_{\Reg}^{\RNum{1}}(f,\phi,\boldsymbol{s})$ and $J_{\Reg}^{\RNum{2}}(f,\phi,\boldsymbol{s})$ will be called the regular orbital integrals of type  $\RNum{1}$ and $\RNum{2}$, respectively.

In the following sections, we will show that $J_{\sm}(f,\phi,\boldsymbol{s})$ is the main term on the geometric side \eqref{eq2.7}, and that
\begin{align*}
J_{\Reg}^{\RNum{1}}(f,\phi,\boldsymbol{s})\equiv J_{\Reg}^{\RNum{2}}(f,\phi,\boldsymbol{s})\equiv 0
\end{align*}
as the absolute norm $N_F(\mathfrak{N})$ becomes sufficiently large.

\section{The Small Cell Orbital Integral}
In this section we will investigate the small cell orbital integral
\begin{multline*}
	J_{\sm}(f,\phi,\boldsymbol{s}):=\int_{F^{\times}\backslash\mathbb{A}_F^{\times}}\int_{(F\backslash\mathbb{A}_F)^2}\int_{[\mathrm{GL}_2]}\sum_{\gamma\in P_0(F)}f\left(\begin{pmatrix}
a& & b\\
& 1& c\\
&& 1
\end{pmatrix}^{-1}\gamma\begin{pmatrix}
y&\\
& 1
\end{pmatrix}\right)\\
 \times\ \  \chi(a)\overline{\psi}(c)|a|^{s_1-\frac{1}{2}}\phi(y)|\det y|^{s_2}d y dbdcd^{\times}a.
\end{multline*}
\begin{prop}\label{prop3.1}
Let notation be as before. Then $J_{\sm}(f,\phi,\boldsymbol{s})$ converges absolutely in $\Re(s_1)\gg 1$ and $\Re(s_2)\gg 1$. Moreover, it admits a holomorphic continuation to $\mathbb{C}^2$, and
\begin{equation}\label{equ3.1}
J_{\sm}(f,\phi,\boldsymbol{0})\gg N_F(\mathfrak{N})^{3},
\end{equation}
where the implied constant depends only on $\chi,$ $\phi$, and $h_v$'s at $v\mid\infty$.
\end{prop}
\begin{proof}
By the change of variables $b\mapsto -ab$, $c\mapsto -c$, and  $y\mapsto ay$, we have
\begin{align*}
J_{\sm}(f,\phi,\boldsymbol{s})=&\int_{\mathbb{A}_F^2}\int_{\mathrm{GL}_2(\mathbb{A}_F)}f\left(\begin{pmatrix}
1& & b\\
& 1& c\\
&& 1
\end{pmatrix}\begin{pmatrix}
y&\\
& 1
\end{pmatrix}\right)\psi(c) \mathcal{P}(y;\phi,\mathbf{s})|\det y|^{s_2}d y dbdc,
\end{align*}
where
\begin{align*}
\mathcal{P}(y;\phi,\mathbf{s}):=\int_{F^{\times}\backslash\mathbb{A}_F^{\times}}\phi\left(\begin{pmatrix}
a\\
&1
\end{pmatrix}y\right)\chi(a)|a|^{s_1+s_2+1/2}d^{\times}a.
\end{align*}

Since $\phi$ is a cusp form, $\mathcal{P}(y;\phi,\mathbf{s})$ converges absolutely for each $y$, defining a continuous function of $y$. Also, the construction of $f$ ensures that
\begin{align*}
f\left(\begin{pmatrix}
1& & b\\
& 1& c\\
&& 1
\end{pmatrix}\begin{pmatrix}
y&\\
& 1
\end{pmatrix}\right)\equiv 0
\end{align*}
unless $y$ lies in a compact set of $\mathrm{GL}_2(\mathbb{A}_F)$, and $(b,c)$ lie in a certain compact set of $\mathbb{A}_F^2$. Hence, the integral $J_{\sm}(f,\phi,\boldsymbol{s})$ converges absolutely for all $\mathbf{s}\in \mathbb{C}^2$.

Utilizing the Fourier expansion of $\phi$ we obtain
\begin{align*}
\mathcal{P}(y;\phi,\mathbf{s})=\int_{\mathbb{A}_F^{\times}}W\left(\begin{pmatrix}
a\\
&1
\end{pmatrix}y\right)\chi(a)|a|^{s_1+s_2+1/2}d^{\times}a,
\end{align*}
where $W$ is the Whittaker function of $\phi$ relative to the additive character $\psi$.  Therefore, we obtain the local factorization:
\begin{align*}
J_{\sm}(f,\phi,\boldsymbol{s})=\prod_{v\leq \infty}J_{\sm,v}(f,\phi,\boldsymbol{s}), \ \ \ \Re(s_1)\gg 1,\ \Re(s_2)\gg 1,
\end{align*}
where
\begin{align*}
J_{\sm,v}(f,\phi,\boldsymbol{s}):=&\int_{\mathrm{GL}_2(F_v)}\int_{F_v^2}f_v\left(\begin{pmatrix}
1& & b_v\\
& 1& c_v\\
&& 1
\end{pmatrix}\begin{pmatrix}
y_v&\\
& 1
\end{pmatrix}\right)\psi_v(c_v) db_vdc_v\\
&\int_{F^{\times}}W_v\left(\begin{pmatrix}
a_v\\
&1
\end{pmatrix}y_v\right)\chi_v(a_v)|a_v|_v^{s_1+s_2+1/2}d^{\times}a_v
|\det y_v|_v^{s_2}dy_v.
\end{align*}

\begin{itemize}
\item Let $v=v_0$. By Lemma \ref{lem2.1}, we have
\begin{align*}
\overline{\psi}_{v}(c_{v})\Vol(K_{v}'^{\circ}[1])\mathbf{1}_{b_{v}, c_{v}\in \mathfrak{p}_{v}^{-1}}\mathbf{1}_{K_{v}'^{\circ}[1]}(y_{v})
\end{align*}	

\begin{align*}
J_{\sm,v}(f,\phi,\boldsymbol{s})=&\Vol(K_{v}'^{\circ}[1])^{-1}\int_{K_{v}'^{\circ}[1]}\int_{\mathfrak{p}_{v}^{-1}}\int_{\mathfrak{p}_{v}^{-1}}db_vdc_v\\
&\int_{F^{\times}}W_v\left(\begin{pmatrix}
a_v\\
&1
\end{pmatrix}y_v\right)\chi_v(a_v)|a_v|_v^{s_1+s_2+1/2}d^{\times}a_vdy_v,
\end{align*}
which simplifies to
\begin{equation}\label{eq3.1}
J_{\sm,v}(f,\phi,\boldsymbol{s})=q_v^2L_v(1+s_1+s_2,\sigma_v\times\chi_v).
\end{equation}
	
\item Let $v\in \Sigma_{\chi_{\fin}}\bigcup \Sigma_{\sigma_{\fin}}$. Then
\begin{align*}
J_{\sm,v}(f,\phi,\boldsymbol{s})=&\frac{1}{\tau(\chi_v)}\sum_{\alpha_v\in (\mathcal{O}_v/\mathfrak{p}_v^{r_{\chi_v}})^{\times}}\chi_v(\alpha)\sum_{\beta_v\in \mathcal{O}_v/\mathfrak{p}_v^{r_{\sigma_v}}}\int_{\mathrm{GL}_2(F_v)}\int_{F_v^2}\int_{F_v^{\times}}\psi_v(c_v) \\
&\mathbf{1}_{K_v}\left(z_vU_1(-\alpha_v\varpi_v^{-r_{\chi_v}})\begin{pmatrix}
1& & b_v\\
& 1& c_v\\
&& 1
\end{pmatrix}\begin{pmatrix}
y_v&\\
& 1
\end{pmatrix}U_2(\beta_v\varpi_v^{-r_{\sigma_v}})\right)d^{\times}z_v\\
&db_vdc_v\int_{F^{\times}}W_v\left(\begin{pmatrix}
a_v\\
&1
\end{pmatrix}y_v\right)\chi_v(a_v)|a_v|_v^{s_1+s_2+1/2}d^{\times}a_v
|\det y_v|_v^{s_2}dy_v,
\end{align*}
where $U_1(-\alpha_v\varpi_v^{-r_{\chi_v}})$ and $U_2(\beta_v\varpi_v^{-r_{\sigma_v}})$ are defined by \eqref{eq2.5}.
	
By the change of variable $y_v\mapsto U_1(\alpha_v\varpi_v^{-r_{\chi_v}})y_v$, along with
\begin{align*}
W_v\left(\begin{pmatrix}
a_v\\
&1
\end{pmatrix}U_1(\alpha_v\varpi_v^{-r_{\chi_v}})y_v\right)=\psi_v(a_v\alpha_v\varpi_v^{-r_{\chi_v}})W_v\left(\begin{pmatrix}
a_v\\
&1
\end{pmatrix}y_v\right),
\end{align*}
we thus simplify $J_{\sm,v}(f,\phi,\boldsymbol{s})$ to
\begin{align*}
&\sum_{\beta_v\in \mathcal{O}_v/\mathfrak{p}_v^{r_{\sigma_v}}}\int_{\mathrm{GL}_2(F_v)}\int_{F_v^2}\mathbf{1}_{K_v}\left(\begin{pmatrix}
1& & b_v\\
& 1& c_v\\
&& 1
\end{pmatrix}\begin{pmatrix}
y_v&\\
& 1
\end{pmatrix}U_2(\beta_v\varpi_v^{-r_{\sigma_v}})\right)\psi_v(c_v) db_vdc_v\\
&\qquad \int_{F^{\times}}W_v\left(\begin{pmatrix}
a_v\\
&1
\end{pmatrix}y_v\right)\chi_v(\varpi_v^{e_v(a_v)})|a_v|_v^{s_1+s_2+1/2}d^{\times}a_v
|\det y_v|_v^{s_2}dy_v.
\end{align*}

Write $y=\begin{pmatrix}
y_{11} & y_{12}\\
y_{21} & y_{22}
\end{pmatrix}$. By the change of variable $\begin{pmatrix}
b_v\\
c_v
\end{pmatrix}\mapsto \begin{pmatrix}
b_v\\
c_v
\end{pmatrix}-y_v\begin{pmatrix}
\beta_v\varpi_v^{-r_{\sigma_v}})\\
0	
\end{pmatrix}$,
we obtain
\begin{align*}
J_{\sm,v}(f,\phi,\boldsymbol{s})=&\sum_{\beta_v\in \mathcal{O}_v/\mathfrak{p}_v^{r_{\sigma_v}}}\int_{K_v}\int_{\mathcal{O}_v^2}\psi_v(c_v-y_{21}\beta_v\varpi_v^{-r_{\sigma_v}}) db_vdc_v\\
&\qquad \int_{F^{\times}}W_v\left(\begin{pmatrix}
a_v\\
&1
\end{pmatrix}y_v\right)\chi_v(\varpi_v^{e_v(a_v)})|a_v|_v^{s_1+s_2+1/2}d^{\times}a_v
dy_v,
\end{align*}
which, due to a straightforward calculation, boils down to
\begin{equation}\label{eq3.2}
J_{\sm,v}(f,\phi,\boldsymbol{s})=q_v^{r_{\sigma_v}}\Vol(K_v[r_{\sigma_v}])L_v(1+s_1+s_2,\sigma_v\times\chi_v).
\end{equation}

\item Let $v<\infty$ and $v\notin \Sigma_{\chi_{\fin}}\bigcup \Sigma_{\sigma_{\fin}}\bigcup \{v_0\}$. Then
\begin{align*}
J_{\sm,v}(f,\phi,\boldsymbol{s})=&\Vol(I_v[e_v(\mathfrak{N})])^{-1}\int_{\mathrm{GL}_2(F_v)}\int_{\mathcal{O}_v^2}\mathbf{1}_{I_v[e_v(\mathfrak{N})]}\left(\begin{pmatrix}
y_v&\\
& 1
\end{pmatrix}\right)\psi_v(c_v) db_vdc_v\\
&\int_{F^{\times}}W_v\left(\begin{pmatrix}
a_v\\
&1
\end{pmatrix}y_v\right)\chi_v(a_v)|a_v|_v^{s_1+s_2+1/2}d^{\times}a_v
|\det y_v|_v^{s_2}dy_v.
\end{align*}

As a consequence, we obtain
\begin{equation}\label{eq3.3}
J_{\sm,v}(f,\phi,\boldsymbol{s})=c_{F_v}\Vol(I_v[e_v(\mathfrak{N})])^{-1}L_v(1+s_1+s_2,\sigma_v\times\chi_v),
\end{equation}
where $c_{F_v}$ is a positive constant depending only on the discriminant of $F$. In particular, $c_{F_v}= 1$ if $v$ is not a ramified place.

\item Let $v\mid\infty$. By definition of $f_v$ (cf. \eqref{eq2.6}), we have
\begin{align*}
\int_{F_v^2}f_v\left(\begin{pmatrix}
1& & b_v\\
& 1& c_v\\
&& 1
\end{pmatrix}\begin{pmatrix}
y_v&\\
& 1
\end{pmatrix}\right)\psi_v(c_v) db_vdc_v\equiv 0
\end{align*}
unless $\|y_v-I_2\|_v\ll \varepsilon$, and
\begin{equation}\label{eq3.4}
\bigg|\int_{\mathrm{GL}_2(F_v)}\int_{F_v^2}f_v\left(\begin{pmatrix}
1& & b_v\\
& 1& c_v\\
&& 1
\end{pmatrix}\begin{pmatrix}
y_v&\\
& 1
\end{pmatrix}\right)\psi_v(c_v) db_vdc_v
|\det y_v|_v^{s_2}dy_v\bigg|\gg 1.
\end{equation}

Let$|\Re(s_1)|, |\Re(s_2)|\leq 1$. By the mean value theorem,
\begin{align*}
\bigg|J_{\sm,v}(f,\phi,\boldsymbol{s})-&\int_{\mathrm{GL}_2(F_v)}\int_{F_v^2}f_v\left(\begin{pmatrix}
1& & b_v\\
& 1& c_v\\
&& 1
\end{pmatrix}\begin{pmatrix}
y_v&\\
& 1
\end{pmatrix}\right)\psi_v(c_v) db_vdc_v\\
&\int_{F^{\times}}W_v\begin{pmatrix}
a_v\\
&1
\end{pmatrix}\chi_v(a_v)|a_v|_v^{s_1+s_2+1/2}d^{\times}a_v
|\det y_v|_v^{s_2}dy_v\bigg|\ll \varepsilon,
\end{align*}
where the implied constant depends on $\varepsilon$, $f_v$ and $\sigma$. In conjunction with \eqref{eq3.4}, we deduce, by taking $\varepsilon$ to be sufficiently small, that
\begin{equation}\label{eq3.5}
J_{\sm,v}(f,\phi,\boldsymbol{s})\gg \Big|\int_{F^{\times}}W_v\begin{pmatrix}
a_v\\
&1
\end{pmatrix}\chi_v(a_v)|a_v|_v^{s_1+s_2+1/2}d^{\times}a_v\Big|\gg 1.	
\end{equation}
Here the last inequality is a consequence of the definition of $W_v$ in \textsection\ref{sec2.1.2}.
\end{itemize}

Therefore, \eqref{equ3.1} follows from \eqref{eq3.1}, \eqref{eq3.2}, \eqref{eq3.3} and \eqref{eq3.5}.
\end{proof}

\section{Regular Orbital Integrals of Type \RNum{1}}
Let $J_{\Reg}^{\RNum{1}}(f,\phi,\boldsymbol{s})$ be the orbital integral defined as in \textsection\ref{sec2.5}. In this section, we aim to show that $J_{\Reg}^{\RNum{1}}(f,\phi,\boldsymbol{s})$ converges absolutely in $\Re(s_1)\gg 1$ and $\Re(s_2)\gg 1$, and moreover,  that it vanishes in this region.
\subsection{Local Analysis}\label{sec4.1}
Let $f$ be the test function constructed in \textsection\ref{sec2.2}.
Let $C=\otimes_{v\leq\infty}C_v$, where $C_v$ is a compact subset of $G(F_v)$ with  $Z(F_v)C_v=\supp f_v$. In particular, $C_v=K_v$ if $v<\infty$ and $v\not\in \Sigma_{\chi_{\fin}}\bigcup \Sigma_{\sigma_{\fin}}$ or $v\nmid v_0\mathfrak{N}$. Hence,
\begin{equation}\label{4.1}
f\left(\begin{pmatrix}
y_1&y_1u+y_2b & b\\
& y_2c & \beta^{-1}+c\\
& y_2& 1
\end{pmatrix}\begin{pmatrix}
k\\
&1
\end{pmatrix}\right)\equiv 0
\end{equation}
unless there exists some $z_v\in F_v^{\times}$ such that
\begin{equation}\label{eq4.2}
z_v\begin{pmatrix}
y_{1,v}&y_{1,v}u_v+y_{2,v}b_v& b_v\\
& y_{2,v}c_v & \beta^{-1}+c_v\\
& y_{2,v}& 1
\end{pmatrix}\in C_v\begin{pmatrix}
K_v'\\
& 1
\end{pmatrix},
\end{equation}
where $K_v'$ is a maximal compact subgroup of $\mathrm{GL}_2(F_v)$.

\subsubsection{Archimedean places}\label{sec4.1.1}
Let $v\mid\infty$. By \eqref{eq4.2}, we obtain
\begin{align*}
\begin{cases}
|z_vy_{1,v}|_v\asymp 1,\ \ |z_v^2y_{2,v}\beta^{-1}|_v\asymp 1\\
|z_v(y_{1,v}u_v+y_{2,v}b_v)|_v\ll 1,\ \ |z_vb_v|_v\ll 1\\
|z_vy_{2,v}|_v\ll 1,\ \ |z_vy_{2,v}c_v|_v\ll 1,\ \ |z_v|_v\ll 1,\ \ |z_v(\beta^{-1}+c_v)|_v\ll 1,
\end{cases}
\end{align*}
where the implied constant depends only on $h_v\in C_c^{\infty}(\mathrm{GL}_3(F_v))$ (cf. \eqref{eq2.6}). So
\begin{equation}\label{equa4.2}
\begin{cases}
|z_v|_v\ll 1,\ \ |z_v^2y_{2,v}|_v\asymp |\beta|_v,\ \ |z_vy_{2,v}|_v\ll 1\ \Rightarrow\ |\beta|_v\ll |z_v|_v\ll 1\\
|z_v(y_{1,v}u_v+y_{2,v}b_v)|_v\ll 1,\ \ |z_vb_v|_v\ll 1\\
|z_vy_{1,v}|_v\asymp 1,\ \ |z_vy_{2,v}c_v|_v\ll 1,\ \ |z_v(\beta^{-1}+c_v)|_v\ll 1.
\end{cases}
\end{equation}

\subsubsection{Non-Archimedean places: ramification}
Let $v\in \Sigma_{\chi_{\fin}}\bigcup \Sigma_{\sigma_{\fin}}$ or $v\mid v_0$. Write $l=e_v(z_v)$. It follows from \eqref{eq4.2} that
\begin{equation}\label{eq4.3}
\begin{cases}
|2l+e_v(y_{2,v})-e_v(\beta)|\ll 1\\
|l+e_v(y_{1,v})|\ll 1\\
l+e_v(y_{2,v})\gg -1\\
l+e_v(y_{2,v})+e_v(c_v)\gg -1\\
l\gg -1\\
l+e_v(\beta^{-1}+ c_v)\gg -1\\
l+e_v(y_{1,v}u_v+y_{2,v}b_v)\gg -1\\
l+e_v(b_v)\gg -1,
\end{cases}
\end{equation}
where the implied constants depend on $\chi$, $\pi_{v_0}$, $\sigma$, and $F$. It follows from \eqref{eq4.3}  that
\begin{equation}\label{eq4.4}
\begin{cases}
e_v(\beta)\gg -1,\ \ -1\ll l\ll 1+|e_v(\beta)|\\
|2l+e_v(y_{2,v})-e_v(\beta)|\ll 1\\
|l+e_v(y_{1,v})|\ll 1\\
l+e_v(y_{2,v})+e_v(c_v)\gg -1\\
l+e_v(\beta^{-1}+ c_v)\gg -1\\
l+e_v(y_{1,v}u_v+y_{2,v}b_v)\gg -1\\
l+e_v(b_v)\gg -1.
\end{cases}
\end{equation}

In particular, for each $\beta$, the constraints \eqref{eq4.4} forces the variables $y_{1,v}$, $y_{2,v}$, $u_v$, $c_v$ and $b_v$ lie in compact sets depending only on $\chi$, $\pi_{v_0}$, $\sigma$, $F$, and $\beta$.

\subsubsection{Non-Archimedean places: unramified places}
Let $v<\infty$ with $v\not\in \Sigma_{\chi_{\fin}}\bigcup \Sigma_{\sigma_{\fin}}$ and $v\nmid v_0$. Write $l=e_v(z_v)$. In this case we have $C_v=K_v[e_v(\mathfrak{N})]$. So \eqref{eq4.2} boils down to
\begin{equation}\label{f4.5}
\begin{cases}
e_v(\beta)=2l+e_v(y_{2,v})\\
l+e_v(y_{2,v})\geq e_v(\mathfrak{N})\\
l+e_v(y_{1,v})=0\\
l+e_v(y_{2,v})+e_v(c_v)\geq 0\\
l\geq 0\\
l+e_v(\beta^{-1}+ c_v)\geq 0\\
l+e_v(y_{1,v}u_v+y_{2,v}b_v)\geq 0\\
l+e_v(b_v)\geq 0.
\end{cases}
\end{equation}
In particular, we have $e_v(\beta)\geq e_v(\mathfrak{N})$. Moreover, when $e_v(\mathfrak{N})\geq 1$, we have
\begin{equation}\label{f4.7}
\begin{cases}
e_v(\beta)=e_v(y_{2,v})\geq e_v(\mathfrak{N}),\ \ e_v(y_{1,v})=0\\
e_v(\beta^{-1}+ c_v)\geq 0,\ \ e_v(y_{1,v}u_v+y_{2,v}b_v)\geq 0,\ \ e_v(b_v)\geq 0.
\end{cases}
\end{equation}

Now we assume $e_v(\mathfrak{N})=0$, so that $C_v=K_v$.  Consider the following scenarios.
\begin{itemize}
\item Suppose $e_v(\beta)=0$. Then
\begin{equation}\label{equa4.4}
\begin{cases}
e_v(\beta)=e_v(y_{2,v})=e_v(y_{1,v})=0\\
e_v(c_v)\geq 0,\ \ e_v(u_v)\geq 0,\ \ e_v(b_v)\geq 0.
\end{cases}
\end{equation}

\item Suppose $e_v(\beta)\geq 1$. We have the following scenarios.
\begin{itemize}
\item Suppose $l=0$. Then
\begin{equation}\label{equ4.8}
\begin{cases}
e_v(\beta)=e_v(y_{2,v})\geq 1\\
e_v(y_{1,v})=0\\
c_v=-\beta^{-1}+\mathcal{O}_v\\
e_v(u_v)\geq 0\\
e_v(b_v)\geq 0.
\end{cases}
\end{equation}

\item Suppose $l\geq 1$. Then
\begin{equation}\label{eq4.9}
\begin{cases}
e_v(\beta)=2l+e_v(y_{2,v})\geq l\\
l+e_v(y_{2,v})= 0\\
l+e_v(y_{1,v})=0\\
l+e_v(y_{2,v})+e_v(c_v)\geq 0\\
l\geq 1\\
l+e_v(\beta^{-1}+ c_v)= 0\\
l+e_v(y_{1,v}u_v+y_{2,v}b_v)\geq 0\\
l+e_v(b_v)\geq 0
\end{cases}\ \ \Rightarrow\ \ \begin{cases}
e_v(\beta)= l\geq 1\\
e_v(y_{2,v})=e_v(y_{1,v})=-l\\
e_v(c_v)\geq 0\\
l+e_v(y_{1,v}u_v+y_{2,v}b_v)\geq 0\\
e_v(b_v)\geq -l.
\end{cases}
\end{equation}
\end{itemize}
\end{itemize}

\subsubsection{Back to Archimedean places}
Suppose that \eqref{4.1} holds. By \eqref{eq4.4} and \eqref{f4.5} there exists a fractional ideal $\mathfrak{I}$, depending only on $F$, $\pi_{v_0}$, $\chi$ and $\sigma$, such that $\beta\in \mathfrak{N}\mathfrak{I}-\{0\}$. So $N(\beta)\gg 1$. Combining this with $|\beta|_v\ll |z_v|_v\ll 1$ for each $v\mid\infty$, we derive that $|z_v|_v\asymp 1$, $v\mid\infty$. Hence, \eqref{equa4.2} reduces to
\begin{equation}\label{eq4.8}
\begin{cases}
|z_v|_v\asymp 1,\ \ |y_{2,v}|_v\ll 1\ \Rightarrow\ |\beta|_v\ll 1\\
|y_{1,v}u_v+y_{2,v}b_v|_v\ll 1,\ \ |b_v|_v\ll 1\\
|y_{1,v}|_v\asymp 1,\ \ |\beta^{-1}+c_v|_v\ll 1.
\end{cases}
\end{equation}

In particular, under the constraint \eqref{4.1}, we have $\beta\in \mathfrak{N}\mathfrak{I}-\{0\}$ with $|\beta|_v\ll 1$ at all $v\mid\infty$. In particular, there are only $O(1)$ such $\beta$'s.

\subsection{Convergence and Vanishing}
\begin{lemma}\label{lem4.1}
Let notation be as before. Let $\phi$ be a cusp form in $\sigma$. Let $\Re(s_1)\geq 10$ and $\Re(s_2)\geq 20$. Then
\begin{align*}
J_{\RNum{1}}:=\int_{F^{\times}\backslash\mathbb{A}_F^{\times}}\int_{(F\backslash\mathbb{A}_F)^2}\int_{\mathrm{GL}_2(\mathbb{A}_F)}\sum_{\alpha,\beta,\gamma\in F}\Big|f\left(X\right)|a|^{s_1-\frac{1}{2}}\phi(y)|\det y|^{s_2}\Big|d y dbdcd^{\times}a<\infty,
\end{align*}
where $X$ refers to the matrix
\begin{equation}\label{eq4.1}
\begin{pmatrix}
a^{-1}& & \\
& 1& \\
&& 1
\end{pmatrix}\begin{pmatrix}
1& & -b\\
& 1& -c\\
&& 1
\end{pmatrix}\begin{pmatrix}
1&  & \\
& 1& \gamma \\
&& 1
\end{pmatrix}w_2\begin{pmatrix}
1& & \alpha \\
& 1& \beta\\
&& 1
\end{pmatrix}\begin{pmatrix}
y&\\
& 1
\end{pmatrix}.
\end{equation}
\end{lemma}
\begin{proof}
Due to the construction of $f_v$ at $v\mid\mathfrak{N}$, we may assume $\beta\in F^{\times}$ in the above inner sum. Taking advantage of
\begin{align*}
\begin{pmatrix}
1& &\alpha\\
& 1& \beta\\
&& 1
\end{pmatrix}\begin{pmatrix}
1& -\delta &\\
& 1& \\
&& 1
\end{pmatrix}=\begin{pmatrix}
1& -\delta &\\
& 1& \\
&& 1
\end{pmatrix}\begin{pmatrix}
1& &\alpha+\delta\beta\\
& 1& \beta\\
&& 1
\end{pmatrix},
\end{align*}
we can preform a series of change of variables $y\mapsto \begin{pmatrix}
1& -\alpha\beta^{-1}\\
& 1
\end{pmatrix}y$, $b\mapsto b+\alpha\beta^{-1}$, $c\mapsto c+\gamma$, and $y\mapsto \begin{pmatrix}
a\\
& 1
\end{pmatrix}y$, in conjunction with the Iwasawa coordinate $y=\begin{pmatrix}
y_1&y_1u\\
& y_2
\end{pmatrix}k$,
to obtain
\begin{align*}
J_{\RNum{1}}=&\int_{F^{\times}\backslash\mathbb{A}_F^{\times}}\int_{(\mathbb{A}_F)^2}\int_{\mathrm{GL}_2(\mathbb{A}_F)}\sum_{\beta\in F^{\times}}\bigg|f\left(\begin{pmatrix}
1& & b\\
& 1& c\\
&& 1
\end{pmatrix}w_2\begin{pmatrix}
y_1&y_1u\\
& y_2& \beta\\
&& 1
\end{pmatrix}\begin{pmatrix}
k\\
&1
\end{pmatrix}\right)\\
&\qquad |a|^{s_1+\frac{1}{2}}\phi\left(\begin{pmatrix}
a&\\
& 1
\end{pmatrix}y\right)|\det y|^{s_2}\bigg|dy dbdcd^{\times}a.
\end{align*} 	

A straightforward calculation leads to
\begin{align*}
&\begin{pmatrix}
1& & b\\
& 1& c\\
&& 1
\end{pmatrix}w_2\begin{pmatrix}
y_1&y_1u\\
& y_2& \beta\\
&& 1
\end{pmatrix}
=\begin{pmatrix}
y_1&y_1u+y_2b & \beta b\\
& y_2c & 1+\beta c\\
& y_2& \beta
\end{pmatrix}.
\end{align*}

Consequently, we obtain, by the change of variable $y_1\mapsto \beta y_1$ and $y_2\mapsto \beta y_2$, that
\begin{align*}
J_{\RNum{1}}=&\sum_{\beta\in F^{\times}}\int_{F^{\times}\backslash\mathbb{A}_F^{\times}}\int_{(\mathbb{A}_F^{\times})^2}\int_{(\mathbb{A}_F)^3}\int_{K'}\bigg|f\left(\begin{pmatrix}
y_1&y_1u+y_2b & b\\
& y_2c & \beta^{-1}+c\\
& y_2& 1
\end{pmatrix}\begin{pmatrix}
k\\
&1
\end{pmatrix}\right)\\
&\qquad |a|^{s_1+\frac{1}{2}}\phi\left(\begin{pmatrix}
ay_1&\\
& y_2
\end{pmatrix}\begin{pmatrix}
1& u\\
& 1
\end{pmatrix}k\right)|y_1y_2|^{s_2}\bigg|dkdu dbdcd^{\times}y_1d^{\times}y_2d^{\times}a.
\end{align*}

Utilizing \eqref{eq4.4}, \eqref{f4.7}, \eqref{equa4.4}, \eqref{equ4.8}, \eqref{eq4.9}, and \eqref{eq4.8} in \textsection\ref{sec4.1}, which characterize the support of the variables $y_1,$ $y_2$, $u$, $b$, $c$ and $a$, in conjunction with the decaying of cusp forms
\begin{align*}
\phi\left(\begin{pmatrix}
ay_1&\\
& y_2
\end{pmatrix}\begin{pmatrix}
1& u\\
& 1
\end{pmatrix}k\right)\ll \min\{1,|ay_1|^{-15}|y_2|^{15}\},
\end{align*}
where the implied constant depends on $\phi$ and $F$,
we obtain $J_{\RNum{1}}<\infty$.
\end{proof}

\begin{remark}
In the above proof the cuspidality of $\phi$ is essential. In fact, $J_{\RNum{1}}$ may not converge if $\phi$ is an Eisenstein series. 	
\end{remark}

\begin{cor}\label{cor4.3}
Let notation be as before. Let $N_F(\mathfrak{N})$ be sufficiently large. Let $\phi$ be a cusp form in $\sigma$. Let $\Re(s_1)\geq 10$ and $\Re(s_2)\geq 20$. Then $J_{\Reg}^{\RNum{1}}(f,\phi,\boldsymbol{s})\equiv 0$.
\end{cor}
\begin{proof}
By definition,
\begin{align*}
J_{\Reg}^{\RNum{1}}(f,\phi,\boldsymbol{s})=&\int_{F^{\times}\backslash\mathbb{A}_F^{\times}}\int_{(F\backslash\mathbb{A}_F)^2}\int_{\mathrm{GL}_2(\mathbb{A}_F)}\sum_{\alpha,\beta,\gamma\in F}f\left(X\right)\\
&\chi(a)\overline{\psi}(c)|a|^{s_1-\frac{1}{2}}\phi(y)|\det y|^{s_2}d y dbdcd^{\times}a,
\end{align*}
where $X$ is defined by \eqref{eq4.1}. By Lemma \ref{lem4.1}, $J_{\Reg}^{\RNum{1}}(f,\phi,\boldsymbol{s})$ is equal to
\begin{align*}
&\sum_{\beta\in F^{\times}}\int_{F^{\times}\backslash\mathbb{A}_F^{\times}}\int_{(\mathbb{A}_F^{\times})^2}\int_{(\mathbb{A}_F)^3}\int_{K'}f\left(\begin{pmatrix}
y_1&y_1u+y_2b & b\\
& y_2c & \beta^{-1}+c\\
& y_2& 1
\end{pmatrix}\begin{pmatrix}
k\\
&1
\end{pmatrix}\right)\\
&\chi(a)\overline{\psi}(c)|a|^{s_1+\frac{1}{2}}\phi\left(\begin{pmatrix}
ay_1&\\
& y_2
\end{pmatrix}\begin{pmatrix}
1& u\\
& 1
\end{pmatrix}k\right)|y_1y_2|^{s_2}dkdu dbdcd^{\times}y_1d^{\times}y_2d^{\times}a,
\end{align*}
which converges absolutely by the control theorem.

By the analysis in \textsection\ref{sec4.1}, there exists a fractional ideal $\mathfrak{I}$, depending only on $F$, $\pi_{v_0}$, $\chi$ and $\sigma$, such that $\beta\in \mathfrak{N}\mathfrak{I}-\{0\}$. Also, by \eqref{eq4.8} we have $|\beta|_v\ll 1$ for all $v\mid\infty$, which contradicts $\beta\in \mathfrak{N}\mathfrak{I}-\{0\}$ if $N_F(\mathfrak{N})$ is sufficiently large. Hence, $J_{\Reg}^{\RNum{1}}(f,\phi,\boldsymbol{s})=0$ in $\Re(s_1)\geq 10$ and $\Re(s_2)\geq 20$.
\end{proof}

\section{Regular Orbital Integrals of Type \RNum{2}}
Let $J_{\Reg}^{\RNum{2}}(f,\phi,\boldsymbol{s})$ be the orbital integral defined in \textsection\ref{sec2.5}. In this section, we aim to show that $J_{\Reg}^{\RNum{2}}(f,\phi,\boldsymbol{s})$ converges absolutely in $\Re(s_1)\gg 1$ and $\Re(s_2)\gg 1$, and moreover,  that it vanishes in this region.

\subsection{Local Analysis}
Let $f$ be the test function constructed in \textsection\ref{sec2.2}.
Let $C=\otimes_{v\leq\infty}C_v$, be as in \textsection\ref{sec4.1}. We have
\begin{align*}
f\left(\begin{pmatrix}
a^{-1}y_1&a^{-1}(y_1u+y_2b) & a^{-1}(\beta^{-1}+b)\\
y_1 & y_1u+y_2c & c\\
& y_2& 1
\end{pmatrix}\begin{pmatrix}
k\\
& 1	
\end{pmatrix}\right)\equiv 0
\end{align*}
unless
\begin{equation}\label{eq5.1}
z_v\begin{pmatrix}
a_v^{-1}y_{1,v}&a_v^{-1}(y_{1,v}u_v+y_{2,v}b_v) & a_v^{-1}(\beta^{-1}+ b_v)\\
y_{1,v} & y_{1,v}u_v+y_{2,v}c_v &  c_v\\
& y_{2,v}& 1
\end{pmatrix}\in C_v\begin{pmatrix}
K_v'\\
& 1
\end{pmatrix}.
\end{equation}

\subsubsection{Archimedean places}
It follows from \eqref{eq5.1} that
\begin{align*}
\begin{cases}
|z_v^3a_v^{-1}\beta^{-1}y_{1,v}y_{2,v}|_v\asymp 1\\
|z_va_v^{-1}y_{1,v}|_v\ll  1,\ \ |z_vy_{1,v}|_v\ll  1\\
|z_va_v^{-1}(y_{1,v}u_v+y_{2,v}b_v)|_v\ll 1,\ \ |z_va_v^{-1}(\beta^{-1}+ b_v)|_v\ll 1\\
|z_vy_{2,v}|_v\ll 1,\ \ |z_vc_v|_v\ll 1,\ \ |z_v|_v\ll 1,\ \ |z_v(y_{1,v}u_v+y_{2,v}c_v)|_v\ll 1,
\end{cases}
\end{align*}
where the implied constant depends only on $h_v\in C_c^{\infty}(\mathrm{GL}_3(F_v))$ (cf. \eqref{eq2.6}). So
\begin{equation}\label{eq5.2}
\begin{cases}
|z_v^3a_v^{-1}\beta^{-1}y_{1,v}y_{2,v}|_v\asymp 1,\ \ |\beta|_v\ll |z_v|_v\ll 1\\
|z_va_v^{-1}y_{1,v}|_v\ll  1,\ \ |z_vy_{1,v}|_v\ll  1\\
|z_va_v^{-1}(y_{1,v}u_v+y_{2,v}b_v)|_v\ll 1,\ \ |z_va_v^{-1}(\beta^{-1}+ b_v)|_v\ll 1\\
|z_vy_{2,v}|_v\ll 1,\ \ |z_vc_v|_v\ll 1,\ \ |z_v(y_{1,v}u_v+y_{2,v}c_v)|_v\ll 1.
\end{cases}
\end{equation}

\subsubsection{Non-Archimedean places: ramification}
Let $v\in \Sigma_{\chi_{\fin}}\bigcup \Sigma_{\sigma_{\fin}}$ or $v\mid v_0$. Write $l=e_v(z_v)$. The constraint \eqref{eq5.1} amounts to
\begin{align*}
\varpi_v^l\begin{pmatrix}
a_v^{-1}y_{1,v}&a_v^{-1}(y_{1,v}u_v+y_{2,v}b_v) & a_v^{-1}(\beta^{-1}+ b_v)\\
y_{1,v} & y_{1,v}u_v+y_{2,v}c_v &  c_v\\
& y_{2,v}& 1
\end{pmatrix}\in C_v,
\end{align*}
which leads to
\begin{equation}\label{eq5.3}
\begin{cases}
|3l-e_v(\beta)-e_v(a_v)+e_v(y_{1,v})+e_v(y_{2,v})|\ll 1\\
l+e_v(y_{2,v})\gg -1,\ \ l\gg -1,\ \ l+e_v(c_v)\gg -1\\
l-e_v(a_v)+e_v(y_{1,v})\gg -1,\ \ l+e_v(y_{1,v})\gg -1\\
l-e_v(a_v)+e_v(y_{1,v}u_v+y_{2,v}b_v)\gg -1\\
l+e_v(y_{1,v}u_v+y_{2,v}c_v)\gg -1\\
l-e_v(a_v)+e_v(\beta ^{-1}+b_v)\gg -1,
\end{cases}
\end{equation}
where the implied constants depend on $\chi$, $\pi_{v_0}$, $\sigma$, and $F$. It follows from \eqref{eq5.3} that
\begin{equation}\label{eq5.4}
\begin{cases}
e_v(\beta)\gg -1,\ \ -1\ll l\ll 1+|e_v(\beta)|\\
|3l-e_v(\beta)-e_v(a_v)+e_v(y_{1,v})+e_v(y_{2,v})|\ll 1 \\
e_v(a_v)\gg l-e_v(\beta),\ \ l+e_v(c_v)\gg -1 \\
-1\ll l+e_v(y_{2,v})\ll 1+|e_v(\beta)|\\
-1\ll l-e_v(a_v)+e_v(y_{1,v})\ll 1+|e_v(\beta)|\\
-1\ll l+e_v(y_{1,v})\ll 1+|e_v(\beta)|\\
l-e_v(a_v)+e_v(y_{1,v}u_v+y_{2,v}b_v)\gg -1\\
l+e_v(y_{1,v}u_v+y_{2,v}c_v)\gg -1\\
l-e_v(a_v)+e_v(\beta ^{-1}+b_v)\gg -1.
\end{cases}
\end{equation}

In particular, for each $\beta$, the constraints \eqref{eq5.4} forces the variables $y_{2,v}$, $u_v$, $c_v$ and $b_v$ lie in compact sets depending only on $\chi$, $\pi_{v_0}$, $\sigma$, $F$, and $\beta$. Moreover, $y_{1,v}$ and $a_v$ satisfy the restrictions
\begin{equation}\label{f5.5}
e_v(a_v)\gg -1-|e_v(\beta)|,\ \ \  e_v(y_{1,v})\gg -1-|e_v(\beta)|.
\end{equation}

\subsubsection{Non-Archimedean places: unramified places}\label{sec5.1.3}
Let $v<\infty$ with $v\not\in \Sigma_{\chi_{\fin}}\bigcup \Sigma_{\sigma_{\fin}}$ and $v\nmid v_0$. Write $l=e_v(z_v)$. We have $C_v=K_v[e_v(\mathfrak{N})]$, and \eqref{eq5.1} reduces to
\begin{equation}\label{eq5.6}
\begin{cases}
e_v(\beta)=3l-e_v(a_v)+e_v(y_{1,v})+e_v(y_{2,v})\\
l+e_v(y_{2,v})\geq e_v(\mathfrak{N})\\
l\geq 0\\
l+e_v(c_v)\geq 0\\
l-e_v(a_v)+e_v(y_{1,v})\geq 0\\
l+e_v(y_{1,v})\geq 0\\
l-e_v(a_v)+e_v(y_{1,v}u_v+y_{2,v}b_v)\geq 0\\
l+e_v(y_{1,v}u_v+y_{2,v}c_v)\geq 0\\
l-e_v(a_v)+e_v(\beta ^{-1}+b_v)\geq 0.
\end{cases}
\end{equation}

When $e_v(\mathfrak{N})\geq 1$, \eqref{eq5.6} simplifies to
\begin{equation}\label{eq5.7}
\begin{cases}
e_v(\beta)=e_v(y_{2,v})\geq e_v(\mathfrak{N})\geq 1\\
e_v(c_v)\geq 0,\ \
e_v(a_v)=e_v(y_{1,v})\geq e_v(\mathfrak{N})\\
-e_v(a_v)+e_v(y_{1,v}u_v+y_{2,v}b_v)\geq 0\\
e_v(\beta ^{-1}+b_v)\geq e_v(\mathfrak{N}).
\end{cases}
\end{equation}

Now we assume $e_v(\mathfrak{N})=0$, so that $C_v=K_v$. Consider the following scenarios.
\begin{itemize}
\item Suppose $e_v(\beta)=0$. Then
\begin{equation}\label{eq5.8}
\begin{cases}
e_v(a_v)=e_v(y_{1,v})\geq 0\,\ \ e_v(y_{2,v})=0,\ \ e_v(c_v)\geq 0\\
e_v(y_{1,v}u_v+y_{2,v}b_v)\geq e_v(a_v),\ \ e_v(\beta ^{-1}+b_v)\geq e_v(a_v).
\end{cases}
\end{equation}

\item Suppose $e_v(\beta)\geq 1$. Then
\begin{equation}\label{eq5.9}
\begin{cases}
e_v(\beta)=3l-e_v(a_v)+e_v(y_{1,v})+e_v(y_{2,v})\geq 1\\
l+e_v(y_{2,v})\geq 0,\ \ l\geq 0\\
l-e_v(a_v)+e_v(y_{1,v})\geq 0\,\ \ e_v(c_v)\geq -l\\
l+e_v(y_{1,v})\geq 0,\ \ l-e_v(a_v)+e_v(y_{1,v}u_v+y_{2,v}b_v)\geq 0\\
l+e_v(y_{1,v}u_v+y_{2,v}c_v)\geq 0,\ \ l-e_v(a_v)+e_v(\beta ^{-1}+b_v)\geq 0.
\end{cases}
\end{equation}

Consider the following scenarios according to the sign of $e_v(a_v)$.
\begin{itemize}
\item Suppose $e_v(a_v)\geq 0$. Then \eqref{eq5.9} reduces to
\begin{align*}
\begin{cases}
2l+e_v(y_{2,v})=e_v(\beta)\geq 1\\
e_v(a_v)\geq 0,\ \ l\geq 0,\ \ l+e_v(y_{2,v})\geq 0\\
l+e_v(c_v-b_v-\beta^{-1})\geq 0\\
l-e_v(a_v)+e_v(y_{1,v})=0\\
l-e_v(a_v)+e_v(y_{1,v}u_v+y_{2,v}b_v)\geq 0\\
l+e_v(y_{2,v})+e_v(c_v-b_v)\geq 0\\
l-e_v(a_v)+e_v(\beta^{-1}+b_v)\geq 0.
\end{cases}
\end{align*}

\begin{itemize}
\item Suppose $l=0$. Then
\begin{equation}\label{eq5.10}
\begin{cases}
e_v(y_{2,v})=e_v(\beta)\geq 1\\
e_v(y_{1,v})=e_v(a_v)\geq 0\\
e_v(c_v)\geq 0\\
e_v(y_{1,v}u_v+y_{2,v}b_v)\geq e_v(a_v)\\
e_v(y_{2,v})+e_v(c_v-b_v)\geq 0\\
e_v(\beta^{-1}+b_v)\geq e_v(a_v)
\end{cases}\ \ \Rightarrow\ \ \begin{cases}
e_v(y_{2,v})=e_v(\beta)\geq 1\\
e_v(y_{1,v})=e_v(a_v)\geq 0\\
e_v(c_v)\geq 0\\
e_v(y_{1,v}u_v+y_{2,v}b_v)\geq e_v(a_v)\\
e_v(\beta^{-1}+b_v)\geq e_v(a_v).
\end{cases}
\end{equation}

\item Suppose $l\geq 1$. Then
\begin{equation}\label{eq5.11}
\begin{cases}
e_v(\beta)=l\geq 1,\ \ e_v(y_{2,v})=-e_v(\beta)\\
e_v(a_v)\geq 0,\ \ e_v(y_{1,v})=e_v(a_v)-e_v(\beta)\\
e_v(y_{1,v}u_v+y_{2,v}b_v)\geq e_v(a_v)-e_v(\beta)\\
e_v(c_v-b_v)\geq 0,\ \ e_v(\beta^{-1}+b_v)\geq e_v(a_v)-e_v(\beta).
\end{cases}
\end{equation}
\end{itemize}

\item Suppose $e_v(a_v)\leq -1$. Then \eqref{eq5.9} reduces to
\begin{equation}\label{eq5.12}
\begin{cases}
2l-e_v(a_v)+e_v(y_{2,v})=e_v(\beta)\\
l+e_v(y_{2,v})\geq 0,\ \ l\geq 0\\
l+e_v(c_v)\geq 0,\ \ l+e_v(y_{1,v})= 0\\
l-e_v(a_v)+e_v(y_{2,v})+e_v(b_v-c_v)\geq 0\\
l+e_v(y_{1,v}u_v+y_{2,v}c_v)\geq 0\\
l-e_v(a_v)+e_v(\beta^{-1}+ b_v-c_v)\geq 0.
\end{cases}
\end{equation}

\begin{itemize}
\item Suppose $l=0$. Then \eqref{eq5.12} becomes
\begin{equation}\label{eq5.13}
\begin{cases}
e_v(\beta)=-e_v(a_v)+e_v(y_{2,v})\geq 1,\ \ e_v(y_{2,v})\geq 0\ \Rightarrow\ e_v(a_v)\geq -e_v(\beta)\\
e_v(c_v)\geq 0,\ \ e_v(y_{1,v})= 0\\
e_v(b_v)\geq -e_v(\beta),\ \ e_v(u_v)\geq 0\\
-e_v(a_v)+e_v(\beta^{-1}+ b_v)\geq 0
\end{cases}.
\end{equation}

\item Suppose $l\geq 1$. Then \eqref{eq5.12} becomes
\begin{equation}\label{eq5.14}
\begin{cases}
e_v(\beta)\geq 1,\ \ e_v(y_{1,v})=e_v(y_{2,v})=-e_v(a_v)-e_v(\beta)\\
e_v(a_v)\geq 1-e_v(\beta),\ \ e_v(c_v)\geq -e_v(a_v)-e_v(\beta)\\
-e_v(a_v)+e_v(b_v-c_v)\geq 0,\ \ e_v(y_{1,v}u_v+y_{2,v}c_v)\geq -e_v(a_v)-e_v(\beta).
\end{cases}
\end{equation}
\end{itemize}
\end{itemize}
\end{itemize}

\subsubsection{Back to Archimedean places}\label{sec5.1.4}
Suppose that \eqref{eq5.1} holds. By \eqref{eq5.7}, \eqref{eq5.8} and \eqref{eq5.9} there exists a fractional ideal $\mathfrak{J}$, depending only on $F$, $\pi_{v_0}$, $\chi$ and $\sigma$, such that $\beta\in \mathfrak{N}\mathfrak{J}-\{0\}$. So $N(\beta)\gg 1$. Combining this with $|\beta|_v\ll |z_v|_v\ll 1$ for each $v\mid\infty$, we derive that $|z_v|_v\asymp 1$, $v\mid\infty$. Hence, \eqref{eq5.2} reduces to
\begin{equation}\label{f5.15}
\begin{cases}
|a_v^{-1}\beta^{-1}y_{1,v}y_{2,v}|_v\asymp 1,\ \ |y_{1,v}|_v\ll  1,\ \ |y_{2,v}|_v\ll 1\ \Rightarrow\ |a_v|_v\ll |\beta|_v\\
|\beta|_v\ll 1,\ \ |a_v^{-1}y_{1,v}|_v\ll  1\\
|a_v^{-1}(y_{1,v}u_v+y_{2,v}b_v)|_v\ll 1,\ \ |a_v^{-1}(\beta^{-1}+ b_v)|_v\ll 1\\
|c_v|_v\ll 1,\ \ |y_{1,v}u_v+y_{2,v}c_v|_v\ll 1.
\end{cases}
\end{equation}

In particular, under the constraint \eqref{eq5.1}, we have $\beta\in \mathfrak{N}\mathfrak{J}-\{0\}$ with $|\beta|_v\ll 1$ at all $v\mid\infty$. In particular, there are only $O(1)$ such $\beta$'s.

\subsection{Convergence and Vanishing}

\begin{lemma}\label{lem5.1}
Let notation be as before. Let $\phi$ be a cusp form in $\sigma$. Let $\Re(s_1)\geq 10$ and $\Re(s_2)\geq \Re(s_1)+10$. Then
\begin{align*}
J_{\RNum{2}}:=\int_{F^{\times}\backslash\mathbb{A}_F^{\times}}\int_{(F\backslash\mathbb{A}_F)^2}\int_{\mathrm{GL}_2(\mathbb{A}_F)}\sum_{\alpha,\beta,\gamma,\delta\in F}\Big|f\left(Y\right)|a|^{s_1-\frac{1}{2}}\phi(y)|\det y|^{s_2}\Big|d y dbdcd^{\times}a
\end{align*}
converges absolutely, where $Y$ refers to the matrix
\begin{equation}\label{eq5.15}
\begin{pmatrix}
a^{-1}& & \\
& 1& \\
&& 1
\end{pmatrix}\begin{pmatrix}
1&& b\\
& 1& c\\
&& 1
\end{pmatrix}\begin{pmatrix}
1&\gamma & \delta\\
& 1& \\
&& 1
\end{pmatrix}w_1w_2\begin{pmatrix}
1& & \alpha \\
& 1& \beta\\
&& 1
\end{pmatrix}\begin{pmatrix}
y&\\
& 1
\end{pmatrix}.
\end{equation}
\end{lemma}
\begin{proof}
By the definition of $f$ in \textsection\ref{sec2.2}, we have $f(Y)\equiv 0$ unless $\beta, \gamma\in F^{\times}$. For $\beta\neq 0$, we have
\begin{align*}
w_1w_2\begin{pmatrix}
1& & \alpha\\
& 1& \beta\\
&& 1
\end{pmatrix}\begin{pmatrix}
1& \alpha\beta^{-1} &\\
& 1& \\
&& 1
\end{pmatrix}=\begin{pmatrix}
1& \\
& 1& \alpha\beta^{-1}\\
&& 1
\end{pmatrix}w_1w_2\begin{pmatrix}
1& & \\
& 1& \beta\\
&& 1
\end{pmatrix}.
\end{align*}

Therefore, by a change of variable, we obtain
\begin{align*}
J_{\RNum{2}}=\int_{F^{\times}\backslash\mathbb{A}_F^{\times}}\int_{(\mathbb{A}_F)^2}\int_{\mathrm{GL}_2(\mathbb{A}_F)}\sum_{\beta,\gamma\in F^{\times}}\Big|f\left(Y'\right)|a|^{s_1-\frac{1}{2}}\phi(y)|\det y|^{s_2}\Big|d y dbdcd^{\times}a,
\end{align*}
where
\begin{align*}
Y':=\begin{pmatrix}
a^{-1}& & \\
& 1& \\
&& 1
\end{pmatrix}\begin{pmatrix}
1&\gamma & b\\
& 1& c\\
&& 1
\end{pmatrix}w_1w_2\begin{pmatrix}
1& & \\
& 1& \beta\\
&& 1
\end{pmatrix}\begin{pmatrix}
y&\\
& 1
\end{pmatrix}.
\end{align*}

Utilizing the algebraic relation $\begin{pmatrix}
1& \gamma\\
& 1
\end{pmatrix}=\begin{pmatrix}
\gamma\\
& 1
\end{pmatrix}\begin{pmatrix}
1& 1\\
& 1
\end{pmatrix}\begin{pmatrix}
\gamma^{-1}\\
& 1
\end{pmatrix}$ we may perform the change of variable $a\mapsto \gamma a$ to write $J_{\RNum{2}}$ as
\begin{align*}
J_{\RNum{2}}=\int_{F^{\times}\backslash\mathbb{A}_F^{\times}}\int_{(\mathbb{A}_F)^2}\int_{\mathrm{GL}_2(\mathbb{A}_F)}\sum_{\beta\in F^{\times}}\Big|f\left(Y''\right)|a|^{s_1-\frac{1}{2}}\phi(y)|\det y|^{s_2}\Big|d y dbdcd^{\times}a,
\end{align*}
where
\begin{align*}
Y'':=\begin{pmatrix}
a^{-1}& & \\
& 1& \\
&& 1
\end{pmatrix}\begin{pmatrix}
1&1 & b\\
& 1& c\\
&& 1
\end{pmatrix}w_1w_2\begin{pmatrix}
1& & \\
& 1& \beta\\
&& 1
\end{pmatrix}\begin{pmatrix}
y&\\
& 1
\end{pmatrix}.
\end{align*}

Write $y=\begin{pmatrix}
y_1\\
& y_2
\end{pmatrix}\begin{pmatrix}
1& u\\
& 1
\end{pmatrix}k$ in the Iwasawa coordinates. Performing the change of variables $y_1\mapsto \beta y_1$ and $y_2\mapsto \beta y_2$, we derive that
\begin{align*}
J_{\RNum{2}}=\sum_{\beta \in F^{\times}}J_{\RNum{2}}(\beta),
\end{align*}
where $J_{\RNum{2}}(\beta)$ is defined by
\begin{multline*}
	\int_{(\mathbb{A}_F^{\times})^3}\int_{(\mathbb{A}_F)^3}\int_{K}\Bigg|f\left(\begin{pmatrix}
a^{-1}y_1&a^{-1}(y_1u+y_2b) & a^{-1}(\beta^{-1}+b)\\
y_1 & y_1u+y_2c & c\\
& y_2& 1
\end{pmatrix}\begin{pmatrix}
k\\
& 1	
\end{pmatrix}\right)\Bigg|\\
\times |a|^{\Re(s_1)-\frac{1}{2}}|\phi(y)||\det y|^{\Re(s_2)}d y dbdcd^{\times}a.
\end{multline*}

Recall the discussion in \textsection\ref{sec5.1.4}, $J_{\RNum{2}}(\beta)\equiv 0$ unless $\beta\in \mathfrak{N}\mathfrak{J}-\{0\}$ and $|\beta|_v\ll 1$ at all $v\mid\infty$. By Dirichlet's unit theorem, the number of such $\beta$'s is finite:
\begin{equation}\label{f5.17}
\sum_{\substack{\beta\in \mathfrak{N}\mathfrak{J}-\{0\}\\ |\beta|_v\ll 1\ v\mid\infty}}1\ll 1.
\end{equation}

Let $\|\phi\|_{\infty}$ be the sup-norm of $\phi$. Since $\phi$ is a cusp form, $\|\phi\|_{\infty}<\infty$. Therefore,
\begin{align*}
J_{\RNum{2}}(\beta)\leq \|\phi\|_{\infty}\prod_{v\leq\infty}J_{\RNum{2},v}(\beta),
\end{align*}
where $J_{\RNum{2},v}(\beta)$ is defined by
\begin{multline*}
	\int_{(F_v^{\times})^3}\!\!\int_{(F)^3}\!\!\int_{K_v}\!\!\Bigg|f_v\!\!\left(\!\!\begin{pmatrix}
a_v^{-1}y_{1,v}&a_v^{-1}(y_{1,v}u_v+y_{2,v}b_v) & a_v^{-1}(\beta^{-1}+b_v)\\
y_{1,v} & y_{1,v}u_v+y_{2,v}c_v & c_v\\
& y_{2,v}& 1
\end{pmatrix}\!\!\begin{pmatrix}
k_v\\
& 1	
\end{pmatrix}\!\!\right)\!\!\Bigg|\\
\times |a_v|_v^{\Re(s_1)-\frac{1}{2}}|\det y_v|_v^{\Re(s_2)}d y_v db_vdc_vd^{\times}a_v.
\end{multline*}

Let $\beta\in \mathfrak{N}\mathfrak{J}-\{0\}$ with $|\beta|_v\ll 1$ at all $v\mid\infty$. Then there exists a finite set $S$ of non-Archimedean places such that $e_v(\beta)=0$ for $v<\infty$ and $v\not\in S$. The set $S$ depends at most on $F$, $h_v$'s, $\pi_{v_0}$, $\chi$ and $\phi$.

\begin{itemize}
\item Let $v<\infty$ with $v\not\in \Sigma_{\chi_{\fin}}\bigcup \Sigma_{\sigma_{\fin}}\bigcup S$ and $v\nmid v_0$. From the discussion in \textsection\ref{sec5.1.3} and the assumption that $v\not\in S$, we have   $e_v(\beta)=0$, and the variables $y_{1,v}$, $y_{2,v}$, $u_v$, $b_v$, and $c_v$ satisfy the constraints:
\begin{align*}
\begin{cases}
e_v(a_v)=e_v(y_{1,v})\geq 0\,\ \ e_v(y_{2,v})=0,\ \ e_v(c_v)\geq 0\\
e_v(y_{1,v}u_v+y_{2,v}b_v)\geq e_v(a_v),\ \ e_v(\beta ^{-1}+b_v)\geq e_v(a_v).
\end{cases}\tag{\ref{eq5.8}}
\end{align*}

As a consequence, we obtain

\begin{multline*}
	J_{\RNum{2},v}(\beta)=\int_{\mathcal{O}_v}dc_v\int_{\mathcal{O}_v-\{0\}}|a_v|_v^{\Re(s_1)+\Re(s_2)-\frac{1}{2}}\int_{\mathcal{O}_v^{\times}}\int_{a_v\mathcal{O}_v^{\times}}\int_{e_v(\beta ^{-1}+b_v)\geq e_v(a_v)}\\
\int_{e_v(y_{1,v}u_v+y_{2,v}b_v)\geq e_v(a_v)}du_vdb_vd^{\times}y_{1,v}d^{\times}y_{2,v}d^{\times}a_v.
\end{multline*}
With a straightforward calculation, we deduce
\begin{equation}\label{eq5.17}
J_{\RNum{2},v}(\beta)=\Vol(\mathcal{O}_v^{\times})^3\Vol(\mathcal{O}_v)^3\zeta_v(1/2+\Re(s_1)+\Re(s_2)).
\end{equation}

\item Let $v\in \Sigma_{\chi_{\fin}}\bigcup \Sigma_{\sigma_{\fin}}\bigcup S$ or $v\mid v_0\mathfrak{N}$. According to \eqref{eq5.4}, \eqref{f5.5}, \eqref{eq5.7}, \eqref{eq5.8},  \eqref{eq5.9}, \eqref{eq5.10}, \eqref{eq5.11}, \eqref{eq5.13}, and \eqref{eq5.14}, along with the assumptions $\Re(s_1)\geq 10$ and $\Re(s_2)\geq \Re(s_1)+10$, we derive that
\begin{equation}\label{eq5.18}
J_{\RNum{2},v}(\beta)\ll 1,
\end{equation}
where the implied constant depends on $s_1$, $s_2$, $F$,  $\pi_{v_0}$, $\chi$ and $\phi$.

\item Let $v\mid\infty$. By \eqref{f5.15}, and the assumptions $\Re(s_1), \Re(s_2)\geq 10$, we obtain
\begin{equation}\label{eq5.19}
J_{\RNum{2},v}(\beta)\ll 1,
\end{equation}
where the implied constant depends on $s_1$, $s_2$, $F$,  $h_v$'s, $\pi_{v_0}$, $\chi$ and $\phi$.
\end{itemize}

Therefore, Lemma \ref{lem5.1} follows from \eqref{f5.17}, \eqref{eq5.17}, \eqref{eq5.18}, and \eqref{eq5.19}.
\end{proof}

\begin{remark}
In the above proof, we have leveraged the fact that the cusp form $\phi$ is bounded.
\end{remark}

\begin{cor}\label{cor5.3}
Let notation be as before. Let $\phi$ be a cusp form in $\sigma$. Let $\Re(s_1)\geq 10$ and $\Re(s_2)\geq \Re(s_1)+10$. Then $J_{\Reg}^{\RNum{2}}(f,\phi,\boldsymbol{s})\equiv 0$. 	
\end{cor}
\begin{proof}
By definition,
\begin{multline*}
	J_{\Reg}^{\RNum{2}}(f,\phi,\boldsymbol{s})=\int_{F^{\times}\backslash\mathbb{A}_F^{\times}}\int_{F\backslash\mathbb{A}_F}\int_{\mathbb{A}_F}\int_{\mathrm{GL}_2(\mathbb{A}_F)}\sum_{\alpha,\beta,\gamma \in F}f\left(Y\right)\\
\chi(a)\psi(c)|a|^{s_1-\frac{1}{2}}\phi(y)|\det y|^{s_2}d y dbdcd^{\times}a,
\end{multline*}
where $Y$ refers to the matrix defined as in \eqref{eq5.15}.  Therefore, it follows from Lemma \ref{lem5.1} and the triangle inequality that $J_{\Reg}^{\RNum{2}}(f,\phi,\boldsymbol{s})$ converges absolutely in the region $\Re(s_1)\geq 10$ and $\Re(s_2)\geq \Re(s_1)+10$. Furthermore,
\begin{multline*}
	J_{\Reg}^{\RNum{2}}(f,\phi,\boldsymbol{s})=\sum_{\substack{\beta\in \mathfrak{N}\mathfrak{J}-\{0\}\\ |\beta|_v\ll 1\ v\mid\infty}}\int_{(\mathbb{A}_F^{\times})^3}\int_{(\mathbb{A}_F)^3}\int_{K}f\left(Y^*\right)\\
\chi(a)\psi(c)|a|^{s_1-\frac{1}{2}}\phi(y)|\det y|^{s_2}d y dbdcd^{\times}a,
\end{multline*}
where $Y^*$ is defined by
\begin{align*}
\begin{pmatrix}
a^{-1}y_1&a^{-1}(y_1u+y_2b) & a^{-1}(\beta^{-1}+b)\\
y_1 & y_1u+y_2c & c\\
& y_2& 1
\end{pmatrix}\begin{pmatrix}
k\\
& 1	
\end{pmatrix}.
\end{align*}

However, when $N_F(\mathfrak{N})$ is large enough, there is no $\beta$ satisfying $\beta\in \mathfrak{N}\mathfrak{J}-\{0\}$ \textit{and} $|\beta|_v\ll 1\ v\mid\infty$. Hence, $J_{\Reg}^{\RNum{2}}(f,\phi,\boldsymbol{s})\equiv 0$.
\end{proof}

\section{Conclusion of the Proof of the Main Theorem}
Let $\Re(s_1)\geq 10$ and $\Re(s_2)\geq \Re(s_1)+10$. Let $I(\mathfrak{N})$ be defined as in \eqref{equ2.1}. By the relative trace formula from \eqref{f2.9} and \eqref{eq2.7}, together with Corollaries \ref{cor4.3} and \ref{cor5.3}, we obtain
\begin{align*}
\sum_{\pi}\sum_{\varphi\in \mathfrak{B}_{\pi}^{I(\mathfrak{N})}}\mathcal{P}_1(1/2+s_1,\pi(f)\varphi,\chi)\mathcal{P}_2(1/2+s_2,\overline{\varphi},\phi)=J_{\sm}(f,\phi,\boldsymbol{s}),
\end{align*}
as an identity of holomorphic functions in the region $\Re(s_1)\geq 10$ and $\Re(s_2)\geq \Re(s_1)+10$.

By Proposition \ref{prop3.1} the function $J_{\sm}(f,\phi,\boldsymbol{s})$ admits an analytic continuation to $\mathbb{C}^2$. Hence, by the uniqueness of the continuation, we obtain
\begin{equation}\label{eq6.1}
\sum_{\pi}\sum_{\varphi\in \mathfrak{B}_{\pi}^{I(\mathfrak{N})}}\mathcal{P}_1(1/2+s_1,\pi(f)\varphi,\chi)\mathcal{P}_2(1/2+s_2,\overline{\varphi},\phi)=J_{\sm}(f,\phi,\boldsymbol{s})
\end{equation}
as an identity of entire functions on $\mathbb{C}^2$. Evaluating \eqref{eq6.1} at $\mathbf{s}=(0,0)$, along with the estimate \eqref{equ3.1} in Proposition \ref{prop3.1}, we deduce
\begin{equation}\label{6.2}
\sum_{\pi}\sum_{\varphi\in \mathfrak{B}_{\pi}^{I(\mathfrak{N})}}\mathcal{P}_1(1/2,\pi(f)\varphi,\chi)\mathcal{P}_2(1/2,\overline{\varphi},\phi)\gg N_F(\mathfrak{N})^{3},
\end{equation}
when $N_F(\mathfrak{N})$ is sufficiently large. It follows from \eqref{6.2} that, similar to the calculation in \cite[\textsection 13]{MRY23}, we obtain 
\begin{equation}\label{6.3}
\sum_{\pi}\sum_{\varphi\in \mathfrak{B}_{\pi}^{K_{\fin}}}\mathcal{P}_1(1/2,\pi(f)\varphi,\chi)\mathcal{P}_2(1/2,\overline{\varphi},\phi)\ll 1,
\end{equation}
where $K_{\fin}:=\prod_{v<\infty}K_v$. Therefore, it follows from \eqref{6.2} and \eqref{6.3} that
\begin{align*}
\sum_{\pi}\sum_{\varphi\in \mathfrak{B}_{\pi}^{I(\mathfrak{N})}-\mathfrak{B}_{\pi}^{K_{\fin}}}\mathcal{P}_1(1/2,\pi(f)\varphi,\chi)\mathcal{P}_2(1/2,\overline{\varphi},\phi)\gg N_F(\mathfrak{N})^{3},
\end{align*}
from which we derive that there is a unitary cuspidal automorphic representation $\pi$ of $\mathrm{PGL}_3/F$ of level $I(\mathfrak{N})$ such that
\begin{align*}
L(1/2,\pi\times\chi)L(1/2,\widetilde{\pi}\times\sigma)\neq 0.
\end{align*}

Therefore, Theorem \ref{M} follows upon replacing $\sigma$ with its contragredient $\widetilde{\sigma}$ and the functional equation of $L(s,\widetilde{\pi}\times\sigma)$.

\section*{Acknowledgements}
Ph.M. was partially supported by the SNF grant 200021\_197045. D.R.  was supported by a grant from the Simons Foundation (award Number: 523557).

	
\bibliography{MRY}

\end{document}